\pdfoutput=1

\documentclass[]{elsarticle}

\usepackage{ifthen}
\usepackage{amsopn}
\usepackage{amsmath}
\usepackage{amssymb}
\usepackage{hyperref}
\usepackage{graphicx}
\usepackage{color}
\usepackage{listings}
\usepackage{subfig}
\usepackage{array}

\newcommand{\ket}[1]{\ensuremath{|#1\rangle}}
\newcommand{\bra}[1]{\ensuremath{\langle#1|}}
\newcommand{\ketbra}[2]{\ensuremath{\ket{#1}\bra{#2}}}

\newcommand{\Tr}{\mathrm{Tr}}
\newcommand{\tr}{\Tr}
\newcommand{\1}{{\rm 1\hspace{-0.9mm}l}}

\newcommand{\ii}{\mathrm{i}}

\newcommand{\ee}{\mathrm{e}}
\newcommand{\C}{\mathbb{C}}

\newtheorem{theorem}{Theorem}

\newtheorem{corollary}[theorem]{Corollary}

\newtheorem{definition}[theorem]{Definition}

\newtheorem{lemma}[theorem]{Lemma}

\newtheorem{proposition}[theorem]{Proposition}
\newtheorem{remark}[theorem]{Remark}

\newenvironment{proof}[1][Proof]{\noindent\textbf{#1.} }{\ \rule{0.5em}{0.5em}}

\usepackage{lineno,hyperref}
\modulolinenumbers[5]
\usepackage{layouts}

\begin{document}
\begin{frontmatter}
\title{Real numerical shadow and generalized B-splines}

\author[virginia]{Charles F. Dunkl}
\author[iitis]{Piotr Gawron}
\author[iitis]{{\L}ukasz Pawela}
\author[iitis,uj]{Zbigniew Pucha{\l}a}
\author[uj,cft]{Karol {\.Z}yczkowski}

\address[virginia]{Department of Mathematics 
    Kerchof Hall, Room 223, P.O. Box 400137 
    University of Virginia, Charlottesville, VA 22904-4137}

\address[iitis]{Institute of Theoretical and Applied Informatics, Polish Academy
    of Sciences, Ba{\l}tycka 5, 44-100 Gliwice, Poland}

\address[uj]{Institute of Physics, Jagiellonian University, ul. Reymonta 4, 
    30-059 Krak{\'o}w, Poland}

\address[cft]{Center for Theoretical Physics, Polish Academy of Sciences, 
al. Lotnik{\'o}w 32/46, 02-668 Warszawa, Poland}

\begin{abstract}
Restricted numerical shadow  $P^X_A(z)$ of an operator $A$ of order $N$
is a probability distribution supported on the numerical range 
$W_X(A)$ restricted to a certain subset $X$ of the set of all pure states --
normalized, one--dimensional vectors in ${\mathbb C}^N$.
Its value at point $z \in {\mathbb C}$ equals to the
probability that the inner product $\langle u |A| u \rangle$ is equal to
$z$, where $u$ stands for a random complex vector from the set $X$
distributed according to the natural measure on this set, induced by the
unitarily invariant Fubini--Study measure. 
For a Hermitian  operator $A$ of order $N$ we derive an explicit formula 
for its shadow restricted to real states, $P^{\mathbb R}_A(x)$,
show relation of this density to the Dirichlet distribution and demonstrate
that it forms a generalization of the $B$--spline.
Furthermore, for operators acting on a space with tensor product structure,
${\cal H}_A \otimes {\cal H}_B$, we analyze the shadow restricted to the
set of maximally entangled states and derive distributions for operators of order
$N=4$. 
\end{abstract}

  %% 47A12 Numerical range, numerical radius
  %% 60B05 Probability measures on topological spaces
  %% 81P16 Quantum state spaces, operational and probabilistic concepts
  %% 33C05 (hypergeometric function 2F1)
  %% 51M15 Real and complex geometry: Geometric constructions
\end{frontmatter}

\section{Introduction}

Consider a complex square matrix $A$ or order $N$. Its standard 
{\sl numerical range} is defined as the following subset of the complex plane,
$$W(A)=\{\bra{u}A\ket{u}:u \in {\mathbb{C}}^N, \|u\|=1\},$$ 
where $u$ denotes a normalized complex vector in ${\cal H}_N$.
Due to the Toeplitz--Hausdorff theorem this set is convex, while for a
Hermitian $A$ it forms an interval belonging to the real axis -- 
see e.g. \cite{Dav,GR1977,Gu04}.

Among numerous generalizations of this notion
we will be concerned with the {\sl restricted numerical range},
\begin{equation}
W_X(A)=\{\bra{u}A\ket{u}:u \in \omega_X \},
\end{equation}
where $\omega_X$ forms a certain subset of the set $\omega$
 of normalized
complex vectors of size $N$. For instance, one can choose $\omega_X$
as the set of all real vectors, and analyze the 'real shadow'
of $A$, denoted by $W_{\mathbb R}(A)$.
 For  an operator $A$ acting on a composed space, 
one studies also numerical range restricted to 
 tensor product states, $W_{\otimes}(A)$,
and the range $W_{E}(A)$
restricted to maximally entangled states \cite{rrange2010,rrange2011}.
 It is worth to emphasize
a crucial difference with respect to the standard notion:
the resticted numerical range needs not to be convex.

In order to define a probability measure supported on 
numerical range of $W(A)$ it is sufficient to consider the 
uniform measure on the sphere $S^{2N-1}$ and the measure induced by the map
$u\to \bra{u}A\ket{u}\in W(A)$ \cite{Shd1,GS2012}.
Alternatively, one considers the space
of quantum states -- equivalence classes of normalized vectors in  
${\mathbb{C}}^N$, which differ by a complex phase, $u \sim e^{i \alpha}u$,
and works with the Haar measure invariant under the action
of the unitary group \cite{shadow1}.
For any matrix $A$ one defines in this way
a probability measure $P_A(z)$ supported on $W(A)$
and  called {\sl numerical shadow} \cite{Shd1} or
{\sl numerical measure} \cite{GS2012}.
The former name is inspired by the fact that
for a normal matrix this measure can be interpreted as
a shadow of an uniformly covered $(N-1)$ dimensional regular
simplex projected on a plane \cite{shadow1,gutkin2013joint}.
In a similar fashion, one can consider numerical shadow of matrices
over the quaternion field, defined as the pushforward measure of
the uniform measure on the sphere $S^{4N-1}$.

Even though several papers on numerical shadow were published
during the last five years \cite{Shd1,GS2012,shadow1}, 
the idea to associate  with the numerical range a probability measure
is much older: as described in a recent review by Holbrook \cite{Hol14}
it goes back to the early papers of Davis \cite{Dav}.

Another variant of the numerical shadow of $A$ can be obtained
by taking random points from the subset $\omega_X$ of the set
of pure states. The corresponding probability measure $P_A^X(z)$, 
called  {\sl restricted numerical shadow} \cite{shadow3},
is by definition supported in restricted numerical range $W_X(A)$.
More generally, one may take an arbitrary probability measure $\mu$
on the set of all pure states (or on the hypersphere $S^{2N-1}$)
and study the measure induced in the numerical range of $A$.

Let $A$ denote a Hermitian matrix of size $N$,
so its numerical range is an interval on the real axis.
The probability distribution generated by the map
$\eta\mapsto\left\langle \eta|A|\eta\right\rangle $, where $\eta$ is a random
point on the unit sphere $\left\{  \eta\in\mathbb{C}^{N}:\sum_{j=1}%
^{N}\left\vert \eta_{j}\right\vert ^{2}=1\right\}  $ equipped with the
unitary-invariant surface measure,
is then equal to the \textit{shadow} of $A$.

It will be convenient to introduce the set ${\Omega}_N$
containing density matrices of order $N$, i.e. 
Hermitian positive definite operators, normalized by the trace condition,
$\rho^*=\rho\ge 0$ with ${\rm Tr}\rho=1$. 
The set ${\Omega}_N$ is convex as it can be considered
as the convex hull of the set of projectors on the pure states of 
dimension $N$ -- see e.g. \cite{bengtsson2006geometry}.
Specifying a measure $\mu$ on the set of density matrices 
allows us to propose a more general definition of numerical shadow.

\begin{definition}
For a given $N \times N$ matrix $A$ and a probability measure $\mu$ 
on the space ${\Omega}_N$ of density matrices of order $N$
we define the numerical shadow of matrix $A$ with respect
to $\mu$ as function on complex numbers
\begin{equation}
\mathcal{P}^{\mu}_A (z) = \int_{\Omega_N} d \mu (\rho) \delta (z - \tr A \rho).
\end{equation}
\end{definition}
The standard numerical shadow,  defined in~\cite{shadow1} and denoted by
$\mathcal{P}_A (z)$, fulfills the above
definition with $\mu$ supported on a pure states invariant to unitary
transformations. 
In fact all restricted numerical shadow presented
in~\cite{shadow3} can be written in the above form.

The main goal of this work is to describe restricted numerical shadow
for several relevant cases. For any symmetric real matrix $A$
we derive its real numerical shadow. To this end we use 
Dirichlet distributions, the properties of which are reviewed
in Sec.~\ref{sec:dirichlet}. We demonstrate that in this case the real shadow 
has the same distribution as a linear combination of 
components of a random vector generated by the Dirichlet distribution.

In Sec.~\ref{sec:splines} we briefly discuss $B$--splines, which correspond to 
complex shadows of Hermitian matrices, and show
their link to generalized Dirichlet distributions
Complex and real shadows of illustrative normal matrices are 
compared in Sec.~\ref{sec:symmetric}, in which some results are obtained
for the case of Hermitian matrices.

Main result of this work --- Theorem~\ref{th:main-theorem}, which characterizes 
the real shadow of real symmetric matrices, is presented in Sec.~\ref{sec:real-shadow}.
Continuity of the shadow at knots is discussed in Sec.~\ref{sec:knots}, while 
formulae for the shadow with respect of real maximally entangled states
for any matrix of size $N=4$ are derived in Sec.~\ref{sec:entangled-shadow}.

\section{The Dirichlet Distribution}\label{sec:dirichlet}

Let $\mathbb{T}_{N-1}$ in $\mathbb{R}^{N-1}$ denotes 
the unit simplex of $N$--point probability distributions,
\begin{equation}
\mathbb{T}_{N-1}:=\left\{  \left(  t_{1},\ldots,t_{N-1}\right)  \in
\mathbb{R}^{N-1}:t_{i}\geq0\forall i,\ \sum_{i=1}^{N-1}t_{i}\leq1\right\} .
\end{equation}
The {\sl Dirichlet} distribution is  a measure  $\mu_{\mathbf{k}}$ on the simplex
 $\mathbb{T}_{N-1}$ parameterized by a vector ${\mathbf{k}}$ of
 $N$ real numbers  $k_{1},\ldots,k_{N}>0$,
\begin{equation}
d\mu_{\mathbf{k}}=\frac{\Gamma\left(  \sum_{i=1}^{N}k_{i}\right)  }%
{\prod_{i=1}^{N}\Gamma\left(  k_{i}\right)  }\prod\limits_{i=1}^{N-1}%
t_{i}^{k_{i}-1}\left(  1-\sum_{i=1}^{N-1}t_{i}\right)  ^{k_{N}-1}dt_{1}\ldots
dt_{N-1}.
\end{equation}
Note that the choice ${\mathbf{k}}=\{1,1,\dots,1\}$ 
gives the flat, Lebesgue measure on the simplex,
while the case 
${\mathbf{k}}=\{1/2,1/2,\dots,1/2\}$ corresponds 
to the statistical distribution -- see e.g. \cite{bengtsson2006geometry}.

Set $\widetilde{k}:=\sum_{i=1}^{N}k_{i}$. For $\alpha\in\mathbb{N}_{0}^{N}$
let $\alpha!:=\prod_{i=1}^{N}\alpha_{i}!,\left\vert \alpha\right\vert
:=\sum_{i=1}^{N}\alpha_{i}$ and $t^{\alpha}:=%
{\textstyle\prod_{i=1}^{N-1}}
t_{i}^{\alpha_{i}}\left(  1-\sum_{i=1}^{N-1}t_{i}\right)  ^{\alpha_{N}}$. 
It follows from the Dirichlet integral that%
\begin{equation}
\label{dirinteg}
\int_{\mathbb{T}_{N-1}}t^{\alpha}d\mu_{\mathbf{k}}=\frac{1}{\left(
\widetilde{k}\right)  _{\left\vert \alpha\right\vert }}%
{\textstyle\prod_{i=1}^{N}}
\left(  k_{i}\right)  _{\alpha_{i}},
\end{equation}
where $\left(  x\right)  _{n}:=%
{\textstyle\prod_{i=1}^{n}} \left(  x+i-1\right)$
denotes the Pochhammer product. 
It satisfies an important asymptotic relationship:
$\left(  z\right)  _{n}=\frac{\Gamma\left(  z+n\right)  }{\Gamma\left(
z\right)  }\sim z^{n}$ as $x\rightarrow\infty$ in the complex half-plane
$\left\{  z:\operatorname{Re}z>0\right\}$.

 Consider the random vector corresponding to choosing a point
 in $\mathbb{T}_{N-1}$ according to
$d\mu_{\mathbf{k}}$ with components $\left(  T_{1},\ldots,T_{N}\right)  $ with
$T_{N}:=1-\sum_{i=1}^{N-1}T_{i}$.
We select an arbitrary vector of $N$ real numbers ordered increasingly,
$a_{1} \leq a_{2}\leq\ldots\leq a_{N}$,
and will be concerned with the probability distribution of their
weighted average,
\begin{align*}
X  & =\sum_{i=1}^{N}a_{i}T_{i}.
\end{align*}
\begin{definition}
The distribution of random variable $X$ will be denoted as
\begin{equation}
\mathcal{D}\left( a_{1},\ldots,a_{N};k_{1},\ldots,k_{N}\right)  .\label{DakN}%
\end{equation}
\end{definition} 
In the case some values of $a_{i}$ are repeated some formulae have to be modified.
It is clear that $a_{1}\leq X\leq a_{N}$. 
There is a moment generating function for $X$.

Let $F\left(  x\right)  $ denote the cumulative distribution function of $X$,
that is,%
\begin{equation}
F\left(  x\right)  :=\Pr\left\{  X\leq x\right\}  .
\end{equation}

\begin{lemma}\label{lemma:exp-value}
\label{mgenF}Suppose $\left\vert r\right\vert <\min_{i}\frac{1}{\left\vert
a_{i}\right\vert }$ then%
\begin{equation}
\mathcal{E}\left[  \left(  1-rX\right)  ^{-\widetilde{k}}\right]  =%
{\textstyle\prod_{i=1}^{N}}
\left(1-ra_{i}\right)  ^{-k_{i}}.
\end{equation}
\end{lemma}

This Lemma, proof of which is provided in \ref{sec:proofs} ,
 implicitly gives an expression for the moments,
 $\mathcal{E}\left[  X^{n}\right]$,
 because $\left(  1-rX\right)  ^{-\widetilde{k}}=\sum_{n=0}^{\infty}%
\frac{\left(  \widetilde{k}\right)  _{n}}{n!}r^{n}X^{n}$.

\begin{corollary}
\label{mean&var}The mean $\mu:=\mathcal{E}\left[  X\right]  =\frac
{1}{\widetilde{k}}\sum_{i=1}^{N}k_{i}a_{i}$ and the variance \linebreak
$\mathcal{E}%
\left[  \left(  X-\mu\right)  ^{2}\right]  =\frac{1}{\widetilde{k}\left(
\widetilde{k}+1\right)  }\sum_{i=1}^{N}k_{i}\left(  a_{i}-\mu\right)  ^{2}$.
\end{corollary}

%Designate this distribution by
%\begin{equation}
%\mathcal{D}\left(  a_{1},\ldots,a_{N};k_{1},\ldots,k_{N}\right)  .\label{DakN}%
%\end{equation}

In the case of $N=2$  it is straightforward to find the density 
for \linebreak
$\mathcal{D}\left(  a_{1}%
,a_{2};k_{1},k_{2}\right)  $,%
\begin{equation}
f\left(  x\right)  =\frac{1}{B\left(  k_{1},k_{2}\right)  \left(  a_{2}%
-a_{1}\right)  ^{k_{1}+k_{2}-1}}\left(  a_{2}-x\right)  ^{k_{1}-1}\left(
x-a_{1}\right)  ^{k_{2}-1},
\end{equation}
where $B(a,b)$ denotes the beta function. 

Let us now return to the general case of an arbitrary
dimension $N$ and  consider the behavior of 
$F\left(  x\right)  $ for $x\notin\left\{ a_{1},a_{2},\ldots,a_{N}\right\}  $. 
Here we require no repeated values in
$\left\{  a_{i}\right\}  $. This involves the intersection of the hyperplane
$\sum_{i=1}^{N}a_{i}t_{i}=x$ with $\mathbb{T}_{N-1}$, which is a convex
polytope whose faces are subsets of $\pi_{i}:=\left\{  t:t_{i}=0\right\}  $
for $1\leq i\leq N-1$, $\pi_{N}:=\left\{  t:%
{\textstyle\sum_{i=1}^{N-1}}
t_{i}=1\right\}  $, and \linebreak
$\pi_{x}:=\left\{  t:\sum_{i=1}^{N}a_{i}%
t_{i}=x\right\}  $. Note that $\sum_{i=1}^{N}a_{i}t_{i}=x$ is equivalent to
\linebreak
$\sum_{i=1}^{N-1}\left(  a_{N}-a_{i}\right)  t_{i}=a_{N}-x$. The vertices of
this polytope come from the intersection of $N-2$ hyperplanes drawn from
$\left\{  \pi_{i}:1\leq i\leq N\right\}  $ with $\pi_{x}$. Introduce the unit
basis vectors $\varepsilon_{i}$ ($1\leq i\leq N-1$) with components $\left(
\delta_{ij}\right)  $. There are two types of vertices:%
\begin{align}
\xi_{i}\left(  x\right)    & =%
{\textstyle\bigcap_{j=1,j\neq i}^{N-1}}
\pi_{j}\cap\pi_{x}=\frac{a_{N}-x}{a_{N}-a_{i}}\varepsilon_{i},1\leq i\leq
N-1;\\
\xi_{ij}\left(  x\right)    & =%
{\textstyle\bigcap_{\ell=1,\ell\neq i,j}^{N-1}}
\pi_{\ell}\cap\pi_{N}\cap\pi_{x}=\frac{a_{j}-x}{a_{j}-a_{i}}\varepsilon
_{i}+\frac{x-a_{i}}{a_{j}-a_{i}}\varepsilon_{j},1\leq i<j\leq N-1.
\end{align}
For any given $x$ some of these vertices are in $\mathbb{T}_{N-1}$ and some
are not. Suppose $a_{M}<x<a_{M+1}$ for some $M$ with $1\leq M<N$, then
$\xi_{i}\left(  x\right)  \in\mathbb{T}_{N-1}$ exactly when $1\leq i\leq M$
since the condition is $0<\frac{a_{N}-x}{a_{N}-a_{i}}<1$, that is, $x>a_{i}$.
Similarly $\xi_{ij}\left(  x\right)  \in\mathbb{T}_{N-1}$ exactly when
$a_{i}<x<a_{j}$, that is, $1\leq i\leq M$ and $M+1\leq j\leq N-1$. Thus the
number of vertices is $M\left(  N-M\right)  $. Each vertex is an extreme
point: to show this one exhibits a linear function $c_{0}+\sum_{i=1}%
^{N-1}c_{i}t_{i}$ which vanishes at the point and is positive at all other
vertices. For $\xi_{i}\left(  x\right)  $ the function $\sum_{j\neq i}t_{j}$
accomplishes this, and for $\xi_{ij}\left(  x\right)  $ use $1-t_{i}-t_{j}$
(this applies to the vertices contained in $\mathbb{T}_{N-1}$, by inspection).

\begin{remark}
\label{polyt}Suppose $a_{M}<x_{1}<x_{2}<a_{M+1}$ then $F\left(  x_{2}\right)
-F\left(  x_{1}\right)  $ is given by the integral of $d\mu_{\mathbf{k}}$ over
a convex polytope with $2M\left(  N-M\right)  $ vertices lying between
parallel hyperplanes. The vertices of the polytope are analytic functions of
$x$ and so $F\left(  x_{2}\right)  -F\left(  x_{1}\right)  $ is analytic in
$x_{2}$ and in the parameters $k_{1},k_{2},\ldots,k_{N}$ (in broad terms,
decompose the integral as a sum of iterated $\left(  N-1\right)  $-fold
integrals each of which has an analytic expression).
\end{remark}

It is straightforward to find the following infinite series expression for
the complementary distribution function
 $1-F\left( x\right)$  for $x \in  ( a_{N-1}, a_{N}]$ --
see ~\ref{sec:proofs}~.
We assumed here that $a_{N-1}<a_{N}$, but other repetitions are allowed. 

\begin{proposition}\label{topF}
For $a_{N-1}<x\leq a_{N}$%
\begin{align*}
1-F\left(  x\right)    & =\frac{\Gamma\left(  \widetilde{k}\right)  \left(
a_{N}-x\right)  ^{\widetilde{k}-k_{N}}}{\Gamma\left(  k_{N}\right)
\Gamma\left(  \widetilde{k}-k_{N}\right)  }\prod_{i=1}^{N-1}\left(
a_{N}-a_{i}\right)  ^{-k_{i}}\\
& \times\sum_{\alpha\in\mathbb{N}_{0}^{N-1}}\frac{\left(  1-k_{N}\right)
_{\left\vert \alpha\right\vert }}{\left(  \widetilde{k}-k_{N}\right)
_{\left\vert \alpha\right\vert +1}}\left(  a_{N}-x\right)  ^{\left\vert
\alpha\right\vert }\prod_{i=1}^{N-1}\frac{\left(  k_{i}\right)  _{\alpha_{i}}%
}{\alpha_{i}!\left(  a_{N}-a_{i}\right)  ^{\alpha_{i}}}.
\end{align*}
\end{proposition}

\begin{corollary}
For $x$ near $a_{N}$ (and $x<a_{N}$) $1-F\left(  x\right)  $ behaves like
$\left(  a_{N}-x\right)  ^{\widetilde{k}-k_{N}}$ and the density $f\left(
x\right)  =\frac{d}{dx}F\left(  x\right)  $ behaves like $\left(
a_{N}-x\right)  ^{\widetilde{k}-k_{N}-1}$.
\end{corollary}

The Dirichlet distribution has a special additivity property which allows us
to restrict to the situation where the $a_{i}$'s are mutually distinct. If two
numbers $a_{i}$'s are equal, say $a_{N-1}=a_{N}$ then $\sum_{i=1}^{N}a_{i}%
t_{i}$ is has the same distribution as $\mathcal{D}\left(  a_{1}%
,\ldots,a_{N-1};k_{1},\ldots,k_{N-1}+k_{N}\right)  $ (see \ref{DakN}). In
other words if $a_{\ell}=a_{\ell+1}=\ldots=a_{\ell+m-1}$ then the distribution
is the same as%
\begin{equation}
\mathcal{D}\left(  a_{1},\ldots,a_{\ell},a_{\ell+m},\ldots,a_{N};k_{1}%
,\ldots,\sum_{i=\ell}^{\ell+m-1}k_{i},k_{\ell+m},\ldots,k_{N}\right)  .
\end{equation}
When each $k_{i}$ is an integer ($\geq1)$ there is a finite sum expression for
the density in terms of piecewise polynomials (splines). This theorem is from
\cite[p.2070]{Shd1}. For simplicity we state the result for the case $0\leq
a_{1}<a_{2}<\ldots<a_{N}$. Let $x_{+}:=\max\left(  0,x\right)  $, with the
convention that $x_{+}^{0}=1$ for $x\geq0$ and $=0$ for $x<0$.

\begin{theorem}
\label{cxshad}Suppose $0\leq a_{1}<a_{2}<\ldots<a_{N}$, $k_{i}\in\mathbb{N}$
for each $i$, then
\begin{equation}
f\left(  x\right)  =\sum_{i=1}^{N}\sum_{j=1}^{k_{i}}\frac{\beta_{ij}}%
{a_{i}B\left(  j,\widetilde{k}-j\right)  }\left(  \frac{x}{a_{i}}\right)
_{+}^{j-1}\left(  1-\frac{x}{a_{i}}\right)  _{+}^{\widetilde{k}-j-1},
\end{equation}
where
\begin{equation}
\prod_{i=1}^{N}\left(  1-ra_{i}\right)  ^{-k_{i}}=\sum_{i=1}^{N}\sum
_{j=1}^{k_{i}}\frac{\beta_{ij}}{\left(  1-ra_{i}\right)  ^{j}}\label{parfrac1}%
\end{equation}
is the partial fraction decomposition (the term with $i=1$ is omitted if
$a_{1}=0$).
\end{theorem}

Observe that each term $\frac{1}{a_{i}B\left(  j,\widetilde{k}-j\right)
}\left(  \frac{x}{a_{i}}\right)  _{+}^{j-1}\left(  1-\frac{x}{a_{i}}\right)
_{+}^{\widetilde{k}-j-1}$ is itself a probability density supported on $0\leq
x\leq a_{i}$. (In the present context $N$ is the number of distinct values,
differing from the statement in \cite{Shd1} where each $k_{i}=1$ and some
values are repeated.) The Theorem shows that the density is a piecewise
polynomial of degree $\widetilde{k}-2$ with discontinuities (in some order
derivative) at the points $\left\{  a_{i}\right\}  $. Because of this spline
interpretation the quantities $a_{i}$ will henceforth be called \textit{knots}.

\section{B--splines and their generalization}\label{sec:splines}

The Dirichlet distribution is closely related to the notion of an
$s$--dimensional $B$--spline introduced by  de Boor \cite{dB1976}.

\begin{definition} Let $\sigma$ be a non-trivial simplex in $\mathbb{R}^{s+k}$. On $\mathbb{R}^s$ we
define the B--spline of order $k$ from $\sigma$ by
\begin{equation}
\mathcal{M}_{k,\sigma}(x_1,\dots,x_s)=\mathrm{vol}(\sigma\cap\{v\in\mathbb{R}^{s+k}:v_j=x_j\,\,(j=1,2,\dots,s)\}).
\end{equation}
\end{definition}
A measure version of the above definition is more useful, thus we define 
the normalized measure on $\mathbb{R}^s$
\begin{equation}
\mathcal{M}_{k,\sigma}(B)= \mathrm{vol}(\sigma\cap\{v\in\mathbb{R}^{s+k}: \{v_j\}_{j=1}^s \in B\}) / \mathrm{vol}(\sigma).
\end{equation}
A non-trivial simplex $\sigma \in \mathbb{R}^{s+k}$ can be written as $W \mathbb{T}_{s+k}$ where $\mathbb{T}_{s+k}$
is a regular simplex and $W$ is an invertible matrix of order $s+k$. The 
simplex is possibly translated if 0 is not a vertex of $\sigma$. We will use 
the notation 

\begin{equation}
\mathcal{M}_{k,W}(B)= \mathrm{vol}(y \in \mathbb{T}_{s+k}: W y \in B\oplus \mathbb{R}^{k}) / \mathrm{vol}(\mathbb{T}_{s+k}).
\end{equation}

Instead of calculating the volume with respect to the flat Lebesgue measure
one can use instead the Dirichlet measure $\mu_{\mathbf{k}}$
with parameters $\mathbf{k}$ instead.
In this way one obtains a generalized notion of $B$-splines.
\begin{equation}
\begin{split}
\mathcal{M}_{k,W}^{(\mathbf{k})}(B)
&=  \mu_{(\mathbf{k})}(y \in \mathbb{T}_{s+k}: W y \in B\oplus \mathbb{R}^{k}) / \mu_{(\mathbf{k})}(\mathbb{T}_{s+k}).
\end{split}
\end{equation}

Therefore, the distribution $\mathcal{D}$ can be viewed as a generalized $B$-spline.
If we take any $N \times N$ invertible matrix $W$ with the first row given by
$\lambda_1\dots\lambda_N$, then a generalized $B$-spline
 is equal to the distribution $\mathcal{D}$ 
\begin{equation}
\mathcal{M}_{N-1,W}^{(\mathbf{k})} = \mathcal{D}(\lambda_1,\dots\lambda_N; \mathbf{k}).
\end{equation}

\section{Shadows of Hermitian and real symmetric matrices}\label{sec:symmetric}

Among several probability measures defined on the set of density matrices
it is convenient to distinguish a class of measures induced by the partial trace
performed on a pure state on the extended system.

We say, that a density matrix  $\rho$ of size $N$
is distributed according to the induced measure
$\mu^{\mathrm{tr}}_{N,K}$~\cite{bengtsson2006geometry} if
\begin{equation}
\rho = \tr_2 \ketbra{\psi}{\psi}, 
\end{equation}
where $\ket{\psi}$ being a uniformly distributed, normalized random vector
in $\mathcal{H}_1 \otimes \mathcal{H}_2 = \mathbb{C}^{N} \otimes
\mathbb{C}^{K}$ and the operation of partial trace  is defined for product 
matrices as $\tr_2 A \otimes B  = A \tr B$ and extended to general case by linearity.
In the case of $K=1$ we obtain a measure on pure states and in the case of $K=N$
we get a Hilbert-Schmidt measure~\cite{bengtsson2006geometry}.

\begin{figure}[tbph]
\centering
\includegraphics[width=1\linewidth]{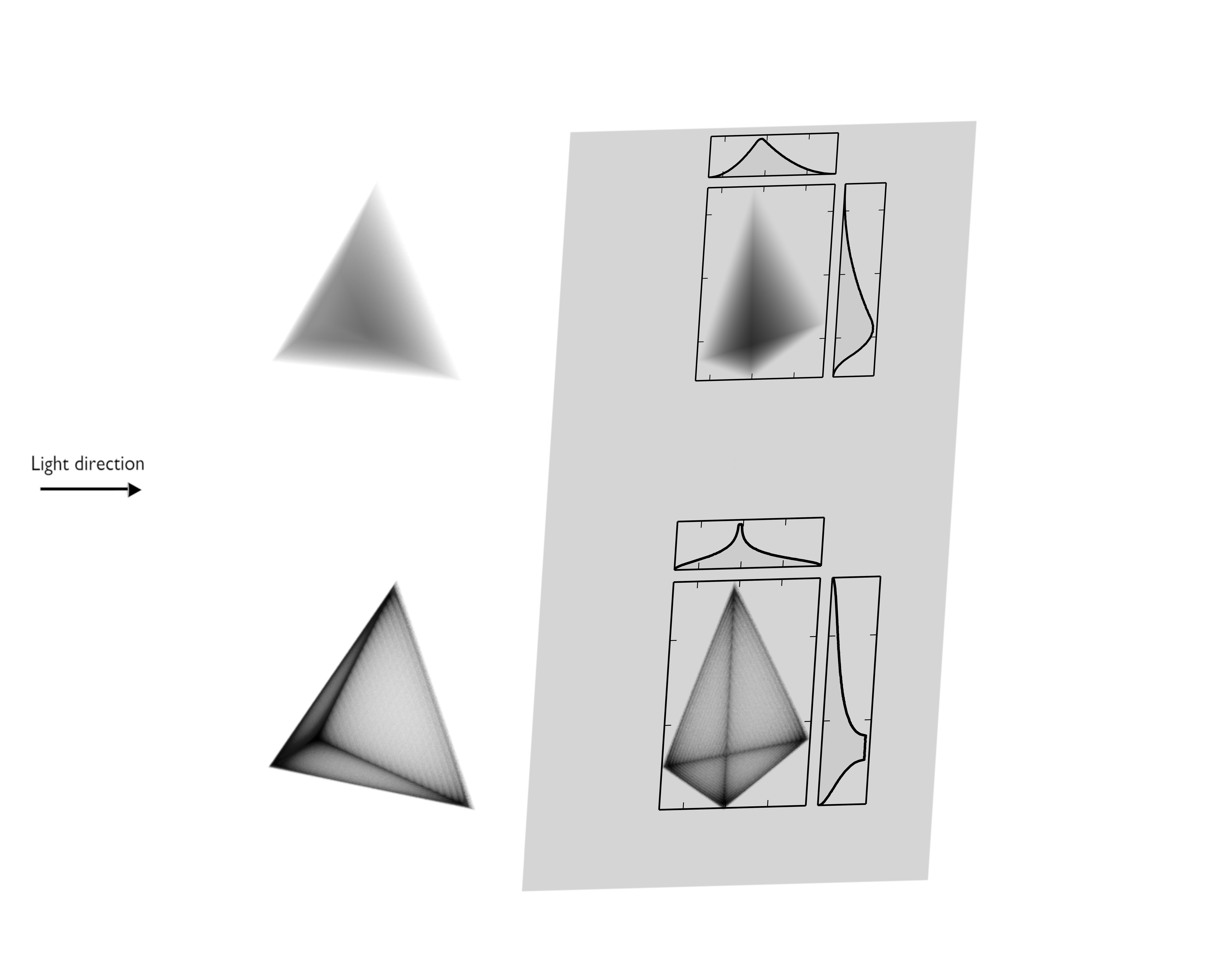}
\caption{Numerical shadows of an illustrative normal
matrix $A$ of order $N=4$.  Upper figure represents the complex
shadow, while the lower one the real shadow
The tetrahedrons on the left are covered with respect to the
uniform and the Dirichlet distribution, respectively. 
If a parallel beam of light is shined upon
them they cast shadows on a plane which coincide with complex and 
real numerical shadows of $A$. Both marginal distributions above and on the right hand side of each shadow, 
correspond to complex/real numerical shadows of Hermitian matrices formed 
from real and imaginary parts of $A$.}
\label{fig:tetrahedrons}
\end{figure}

In paper~\cite{Shd1} we showed that the (complex) shadow of 
a Hermitian matrix $A$ with
eigenvalues $(\lambda_{1},\ldots,\lambda_{N})$ (counted with multiplicity) has
the distribution
\begin{equation}
\mathcal{P}_{A} =  \mathcal{D}\left( 
\lambda_{1},\ldots,\lambda_{N};1,\ldots,1\right).
\label{compl}
\end{equation}
From Corollary
\ref{mean&var} the mean is $\mu=\frac{1}{N}\sum_{j=1}^{N}\lambda_{j}=\frac
{1}{N}\mathrm{Tr}A$ and the variance is $\frac{1}{N\left(  N+1\right)  }%
\sum_{j=1}^{N}\left(  \lambda_{j}-\mu\right)  ^{2}$.

In analogy to the standard shadow (\ref{compl}) one can introduce the
{\sl mixed states shadow} \cite{shadow1}. 
For a Hermitian matrix $A$ the mixed shadow induced by a distribution
$\mu^{\mathrm{tr}}_{N,K}$ reads
\begin{equation}
\mathcal{P}^K_{A} =  
\mathcal{D}(\lambda_1,\dots,\lambda_N; K, \dots, K).
\label{complmixed}
\end{equation}
This follows directly from the definition of a partial trace and the additivity
property of a Dirichlet distribution. As a special case we obtain, that the
mixed numerical shadow with respect to flat Hilbert Schmidt distribution is
given by 
$\mathcal{P}^N_{A} =\mathcal{D}(\lambda_1,\dots,\lambda_N;N, \dots, N)$. 
We can calculate
mean and variance for mixed numerical shadow induced by
$\mu^{\mathrm{tr}}_{N,K}$, using Corollary \ref{mean&var} we have
$\mu=\frac{1}{N}\sum_{j=1}^{N}\lambda_{j}=\frac{1}{N}\mathrm{Tr}A$
and the variance is $\frac{1}{N\left(  N K + 1\right)  } \sum_{j=1}^{N}
\left( \lambda_{j}-\mu\right)  ^{2}$.

Let us now return to the main subject of the paper - the shadow
$\mathcal{P}_A^{\mathbb{R}}$ of a matrix $A$ of order $N$
with respect to the set of real pure states in $\mathbb{R}^{N}$.
It is briefly called the {\sl real shadow} \cite{shadow3},
and for a real symmetric matrix $A$ it 
can be related to the Dirichlet distribution,

\begin{equation}
\mathcal{P}_A^{\mathbb{R}} =  
\mathcal{D}
\left(  \lambda_{1},\ldots,\lambda_{N};\frac{1}{2},\ldots,\frac{1}{2}\right)
\label{realshad}
\end{equation}
where  $\lambda_{1},\ldots,\lambda_{N}$ denotes the
eigenvalues  of $A$ counted with multiplicity.
The mean value is $\mu=\frac{1}{N}\sum_{j=1}^{N}%
\lambda_{j}$ and the variance is $\frac{2}{N\left(  N+2\right)  }\sum
_{j=1}^{N}\left(  \lambda_{j}-\mu\right)  ^{2}$. 

In a close analogy to the complex case, one can also consider the shadow 
with respect to real mixed states obtained by an induced measure
 $\mu_{N,K}^{\rm tr}$. For any real 
symmetric matrix  $A$ this leads to the distribution $\mathcal{D}$,
with all indices equal to $K/2$. Thus the real shadow is obtained for $K=1$,
as required.

Henceforth we will concentrate on the distributions $\mathcal{D}\left(
a_{1},\ldots,a_{N};k,\ldots,k\right)  $ with pairwise distinct knots $a_{i}$.
For integer $k$ we have the interpretation as the shadow of the $Nk\times Nk$
Hermitian matrix $A\oplus\ldots\oplus A$ ($k$ summands) where the eigenvalues
of $A$ are $a_{1},\ldots,a_{N}$, 
 or the mixed numerical shadow induced by the measure $\mu^{\mathrm{tr}}_{N,k}$.
 We consider the distribution
as an analytic function of $k$, for $\operatorname{Re}k>0$, and will find more
information by extrapolating from the known formulas for integer $k$.
Start with
finding explicit values of the coefficients $\left\{  \beta_{ij}\right\}  $ in
Theorem~\ref{cxshad}.

\begin{lemma}\label{kparfrac}
Suppose $\left\{  a_{1},a_{2},\ldots,a_{N}\right\}  $ consists
of pairwise distinct nonzero real numbers and $k=1,2,3,\ldots$ then%
\begin{equation}
\prod_{i=1}^{N}\left(  1-ra_{i}\right)  ^{-k}=\sum_{i=1}^{N}\sum_{m=0}%
^{k-1}\frac{\left(  -1\right)  ^{m}a_{i}^{\left(  N-1\right)  k}}{\left(
1-ra_{i}\right)  ^{k-m}}\sum_{\alpha\in\mathbb{N}_{0}^{N},\left\vert
\alpha\right\vert =m,\alpha_{i}=0}\frac{1}{\alpha!}\prod_{j=1,j\neq i}%
^{N}\frac{\left(  k\right)  _{\alpha_{j}}a_{j}^{\alpha_{j}}}{\left(
a_{i}-a_{j}\right)  ^{k+\alpha_{j}}}%
\end{equation}
\end{lemma}
The proof is provided in~\ref{sec:proofs}.

Thus the formula in Theorem~\ref{cxshad} is completely symmetric in $\left(
a_{1},a_{2},\ldots,a_{N}\right) $, independent of the ordering. This is an
ingredient in the derivation of the differential equation satisfied by the
density. 

Consider the case of a symmetric matrix of size $N=3$.
 Then the density for its real shadow has an expression in
terms of a $_{2}F_{1}$-hypergeometric function which solves a certain
second-order differential equation. Suppose $a_{1}=0$, so formulas
(\ref{parfrac1}) and (\ref{kparfrac}) (change $j$ to $k-m$) 
read for $x\in (a_{2}, x\leq a_{3}]$
\begin{equation}
f\left(  x\right)  =\frac{x^{k-1}\left(  a_{3}-x\right)  ^{2k-1}}{B\left(
k,2k\right)  a_{3}^{2k-1}\left(  a_{3}-a_{2}\right)  ^{k}}~_{2}F_{1}\left(
\genfrac{}{}{0pt}{}{k,1-k}{2k}%
;\frac{a_{2}\left(  a_{3}-x\right)  }{x\left(  a_{3}-a_{2}\right)  }\right)
;\label{N3&2F1}%
\end{equation}
and the series converges for any $k>0$.

Let us now return to the generalized case of an arbitrary matrix order $N$, for
which condition $a_{1}<a_{2}<\ldots<a_{N}$ holds.
Basing on computational experiments we are
in position to formulate a generalization 
valid for small $N$ and integers $k$.
Set $P_{N}\left(  x\right)  =\prod_{i=1}^{N}\left(  x-a_{i}\right)  $ and
define a differential operator $\mathcal{T}_{k}$ of order $N-1$ (with
$\partial:=\frac{d}{dx}$) by%
\begin{equation}
\mathcal{T}_{k}:=P_{N}\left(  x\right)  \partial^{N-1}+\sum_{j=1}^{N-1}\left(
-1\right)  ^{j}\frac{N-j}{N}\frac{\left(  N\left(  k-1\right)  \right)  _{j}%
}{j!}\partial^{j}P_{N}\left(  x\right)  \partial^{N-1-j}.
\end{equation}
The differential equation $\mathcal{T}_{k}f\left(  x\right)  =0$ has regular
singular points at the knots. We will show that the density function of
$\mathcal{D}\left(  a_{1},\ldots,a_{N};k,\ldots,k\right)  $ satisfies this
equation at all $x\notin\left\{  a_{1},\ldots,a_{N}\right\}  $, first for
integer $k$ then for $k>0$. The idea is to verify the equation for the
interval $\left(  a_{N-1},a_{N}\right)  $ by use of Proposition \ref{topF} and
then use the symmetry property of Theorem~\ref{cxshad} to extend the result to
all intervals $\left(  a_{i},a_{i+1}\right)  $.

\begin{lemma}
For arbitrary $a,b,c$ and $n=1,2,\ldots$%
\begin{equation}
\sum_{j=0}^{n}\frac{a+j}{a}\frac{\left(  -n\right)  _{j}\left(  b\right)
_{j}}{\left(  c\right)  _{j}~j!}=\frac{\left(  c-b\right)  _{n-1}}{a~\left(
c\right)  _{n}}\left(  a\left(  c-b+n-1\right)  -nb\right)  .
\end{equation}

\end{lemma}

\medskip 
\begin{proof}
Expand the sum as
\begin{align}
\sum_{j=0}^{n}\frac{\left(  -n\right)  _{j}\left(  b\right)  _{j}}{\left(
c\right)  _{j}~j!}+\frac{1}{a}\sum_{j=1}^{n}\frac{j\left(  -n\right)
_{j}\left(  b\right)  _{j}}{\left(  c\right)  _{j}~j!} &  =\sum_{j=0}^{n}%
\frac{\left(  -n\right)  _{j}\left(  b\right)  _{j}}{\left(  c\right)
_{j}~j!}-\frac{nb}{ac}\sum_{i=0}^{n-1}\frac{\left(  1-n\right)  _{i}\left(
b+1\right)  _{i}}{\left(  c+1\right)  _{i}~i!}\\ \nonumber
&  =\frac{\left(  c-b\right)  _{n}}{\left(  c\right)  _{n}}-\frac{nb}{ac}%
\frac{\left(  c-b\right)  _{n-1}}{\left(  c+1\right)  _{n-1}},
\end{align}
by the Chu-Vandermonde sum.
\end{proof}
\medskip 

Since we intend to work with polynomials in $x-a_{N}$ we set $y:=x-a_{N}$.
Start the verification by replacing $P_{N}\left(  x\right)  $ by $y^{n}$ and
apply the resulting operator to $\left(  -y\right)  ^{c}$ (for $0\leq n\leq
N-1$ and generic $c$ (leaving open the possibility of $c$ being a noninteger
and $y<0)$. At times we use the Pochhammer symbol with a negative index: for
$m=1,2,3,\ldots$ let $\left(  c\right)  _{-m}=1/\left(  c-m\right)  _{m}$, 
so that $\left(  c\right)  _{-m}\left(  c-m\right)  _{m}=\left(  c\right)
_{0}=1$. Note that $\partial^{j}\left(  -y\right)  ^{c}=\left(
-c\right)  _{j}\left(  -y\right)  ^{c-j}$, so the result follows
\begin{align}
\nonumber
&  \sum_{j=0}^{N-1}\left(  -1\right)  ^{j}\frac{N-j}{N}\frac{\left(  N\left(
k-1\right)  \right)  _{j}}{j!}\left(  -1\right)  ^{j}\left(  -n\right)
_{j}y^{n-j}\left(  -c\right)  _{N-1-j}\left(  -y\right)  ^{c-N+1+j}\\ 
&  =\left(  -1\right)  ^{n}\left(  -y\right)  ^{c+n-N+1}\left(  -c\right)
_{N-1}\sum_{j=0}^{n}\frac{j-N}{-N}\frac{\left(  -n\right)  _{j}\left(
N\left(  k-1\right)  \right)  _{j}}{j!\left(  c+2-N\right)  _{j}}\\
\nonumber
&  =\left(  -1\right)  ^{n}\left(  -y\right)  ^{c+n-N+1}\frac{\left(
-c\right)  _{N-1}}{\left(  c+2-N\right)  _{n}}\left(  c+2-Nk\right)
_{n-1}\left(  c+1-\left(  N-n\right)  k\right)  \\
\nonumber
&  =\left(  -y\right)  ^{c+n-N+1}\left(  -c\right)  _{N-1-n}\left(
c+2-Nk\right)  _{n-1}\left(  c+1-\left(  N-n\right)  k\right)  .
\end{align}
In the special case $n=N$ we obtain $-\left(  -y\right)  ^{c+1}\left(
c+2-Nk\right)  _{N-1}$ since $\left(  -c\right)  _{-1}=-1/\left(  1+c\right)
$. For $n=0$ the result is zero. The calculations used the reversal $\left(
a\right)  _{m-j}=\left(  -1\right)  ^{j}\dfrac{\left(  a\right)  _{m}}{\left(
1-m-a\right)  _{j}}$ and the Lemma with $a,b,c$ replaced by $-N,Nk-N$ and $c+2-N$,
respectively.
The upper limit of summation is $n$ because $n\leq N$. 
To proceed further we introduce:
\begin{align}
A\left(  N,k,n,c\right)   &  :=\left(  -c\right)  _{N-1-n}\left(
c+2-Nk\right)  _{n-1}\left(  c+1-\left(  N-n\right)  k\right)  ,1\leq
n<N;\label{defAN}\\
A\left(  N,k,N,c\right)   &  :=-\left(  c+2-Nk\right)  _{N-1}.\nonumber
\end{align}
Next $P_{N}\left(  x\right)  =y\prod_{i=1}^{N-1}\left(  y-\left(  a_{i}%
-a_{N}\right)  \right)  =\sum_{j=0}^{N-1}\left(  -1\right)  ^{N-1-j}%
e_{N-1-j}y^{j+1}$ where $e_{m}$ denotes the elementary symmetric polynomial of
degree $m$ in
\linebreak
$\left\{  a_{1}-a_{N},\ldots,a_{N-1}-a_{N}\right\}  $, $0\leq
m\leq N-1$. Thus
\begin{equation}
\mathcal{T}_{k}\left(  \left(  -y\right)  ^{c}\right)  =\sum_{j=1}^{N}\left(
-1\right)  ^{N-j}e_{N-j}A\left(  N,k,j,c\right)  \left(  -y\right)
^{c-N+j+1}.\label{Tkxc}%
\end{equation}
Up to a multiplicative constant, not relevant in this homogeneous equation,
the density in $a_{N-1}<x<a_{N}$ is given by
\begin{equation}
f_{0}\left(  x\right)  =-\partial\sum_{\alpha\in\mathbb{N}_{0}^{N-1}}%
\frac{\left(  1-k\right)  _{\left\vert \alpha\right\vert }}{\left(  \left(
N-1\right)  k\right)  _{\left\vert \alpha\right\vert +1}}\left(  -y\right)
^{\left\vert \alpha\right\vert +\left(  N-1\right)  k}\prod_{i=1}^{N-1}%
\frac{\left(  k\right)  _{\alpha_{i}}}{\alpha_{i}!\left(  a_{N}-a_{i}\right)
^{\alpha_{i}}}.
\end{equation}
The series terminates at $\left\vert \alpha\right\vert =k-1$. Define symmetric
polynomials $S_{m}\left(  k;a\right)  $ in $\left\{  a_{1}-a_{N}%
,\cdots,a_{N-1}-a_{N}\right\}  $ (note the reversal to $a_{i}-a_{N}$) by
\begin{equation}
\sum_{m=0}^{\infty}S_{m}\left(  k;a\right)  r^{m}=\prod_{i=1}^{N-1}\left(
1-\frac{r}{a_{i}-a_{N}}\right)  ^{-k},
\end{equation}
convergent for $\left\vert r\right\vert <a_{N}-a_{N-1}$, then%
\begin{equation}
f_{0}\left(  x\right)  =\sum_{m=0}^{k-1}\frac{\left(  1-k\right)  _{m}%
}{\left(  \left(  N-1\right)  k\right)  _{m}}\left(  -y\right)  ^{\left(
N-1\right)  k+m-1}\left(  -1\right)  ^{m}S_{m}\left(  k;a\right)  ,
\end{equation}
and%
\begin{align}
\nonumber
\mathcal{T}_{k}f_{0}\left(  x\right)    & =\sum_{m=0}^{k-1}\sum_{j=1}%
^{N}\left(  -1\right)  ^{N-j}e_{N-j}\frac{\left(  1-k\right)  _{m}}{\left(
\left(  N-1\right)  k\right)  _{m}}S_{m}\left(  k;a\right)  \left(  -1\right)
^{m}\\
& \times A\left(  N,k,j,\left(  N-1\right)  k+m-1\right)  \left(  -y\right)
^{\left(  N-1\right)  k+m-N+j}\\
\nonumber
& =\sum_{\ell=1}^{N+k-1}\left(  -y\right)  ^{\left(  N-1\right)  k-N+\ell}%
\sum_{j=1}^{\min\left(  N,\ell\right)  }\left(  -1\right)  ^{N-j}e_{N-j}%
\frac{\left(  1-k\right)  _{\ell-j}}{\left(  \left(  N-1\right)  k\right)
_{\ell-j}}\\
\nonumber
& \times\left(  -1\right)  ^{\ell-j}S_{\ell-j}\left(  k;a\right)  A\left(
N,k,j,\left(  N-1\right)  k+\ell-j-1\right)  .
\end{align}
It is required to show that the $j$-sum vanishes for each $\ell$. At
$\ell=1,j=1$ there is only one term and $A\left(  N,k,1,\left(  N-1\right)
k-1\right)  =0$. Replace $A\left(  \cdot\right)  $ by its definition
(\ref{defAN}) and simplify%
\begin{align}
\nonumber
& \frac{\left(  1-k\right)  _{\ell-j}}{\left(  \left(  N-1\right)  k\right)
_{\ell-j}}\left(  1+j-\ell-\left(  N-1\right)  k\right)  _{N-1-j}\left(
1-k+\ell-j\right)  _{j-1}\left(  \ell-j+k\left(  j-1\right)  \right)  \\
& =\left(  -1\right)  ^{\ell-j}\left(  1-k\right)  _{\ell-1}\left(
\ell-j+k\left(  j-1\right)  \right)  \frac{\left(  1+j-\ell-\left(
N-1\right)  k\right)  _{N-1-j}}{\left(  1+j-\ell-\left(  N-1\right)  k\right)
_{\ell-j}}\\
\nonumber
& =\left(  -1\right)  ^{\ell-j}\left(  1-k\right)  _{\ell-1}\left(
\ell-j+k\left(  j-1\right)  \right)  \frac{\left(  1-\ell-\left(  N-1\right)
k\right)  _{N-1}}{\left(  1-\ell-\left(  N-1\right)  k\right)  _{\ell}}.
\end{align}
The denominator does not vanish because $\left(  N-1\right)  k>0$.
Taking out the factors depending only on $\ell$ the $j$-sum becomes
\begin{equation}
\sum_{j=1}^{\min\left(  N,\ell\right)  }\left(  -1\right)  ^{j}\left(
\ell-j+k\left(  j-1\right)  \right)  e_{N-j}S_{\ell-j}\left(  k;a\right)
.\label{jsum}%
\end{equation}
There is a recurrence relation for $S_{m}\left(  k;a\right)  $; the elementary
symmetric function of degree $m$ in $\left\{  \frac{1}{a_{1}-a_{N}}%
,\ldots,\frac{1}{a_{N-1}-a_{N}}\right\}  $ equals $\frac{e_{N-1-m}}{e_{N-1}}$
for $0\leq m\leq N-1$. The generating function of $\left\{  S_{m}\left(
k;a\right)  \right\}  $ is $g\left(  r\right)  ^{-k}$ where%
\begin{equation}
g\left(  r\right)  :=\prod_{j=1}^{N-1}\left(  1-\frac{r}{a_{j}-a_{N}}\right)
=\sum_{i=0}^{N-1}\left(  -1\right)  ^{i}\frac{e_{N-1-i}}{e_{N-1}}r^{i}.
\end{equation}
Extract the coefficient of $r^{m}$ in the following equation%
\begin{align}
g\left(  r\right)  \frac{\partial}{\partial r}\left[  g\left(  r\right)
^{-k}\right]    & =-k\left(  \frac{\partial}{\partial r}g\left(  r\right)
\right)  g\left(  r\right)  ^{-k}\\
\nonumber
\sum_{i=0}^{N-1}\left(  -1\right)  ^{i}\frac{e_{N-1-i}}{e_{N-1}}r^{i}%
\sum_{j=0}^{\infty}jS_{j}\left(  k;a\right)  r^{j-1}  & =-k\sum_{i=0}%
^{N-1}i\left(  -1\right)  ^{i}\frac{e_{N-1-i}}{e_{N-1}}r^{i-1}\sum
_{j=0}^{\infty}S_{j}\left(  k;a\right)  r^{j},
\end{align}
to obtain%
\begin{equation}
\sum_{i=0}^{\min\left(  m+1,N-1\right)  }\left(  -1\right)  ^{i}\left(
m-i+1+ki\right)  \frac{e_{N-1-i}}{e_{N-1}}S_{m-i+1}\left(  k;a\right)  =0.
\end{equation}
Now set $i=j-1$ and $m=\ell-2$ (recall the case $\ell=1$ was already done) to
show that the expression in (\ref{jsum}) vanishes.

\medskip 
\begin{theorem}
\label{shadeq1}Suppose $k=1,2,3,\ldots$ then the density $f\left(  x\right)  $
of \linebreak
$\mathcal{D}\left(  a_{1},\ldots,a_{N};k,\ldots,k\right)  $ 
satisfies the
linear differential equation $\mathcal{T}_{k}f\left(  x\right)  =0$ at all
$x\notin\left\{  a_{1},\ldots,a_{N}\right\}  $.
\end{theorem}

\begin{proof}
Assume first that $a_{1}>0$. The above argument showed that $\mathcal{T}%
_{k}f\left(  x\right)  =0$ for $a_{N-1}<x<a_{N}$. On this interval $f\left(
x\right)  $ is a constant multiple of%
\begin{equation}
p_{N}\left(  x\right)  :=\sum\limits_{j=1}^{k}\beta_{Nj}\dfrac{1}{B\left(
Nk-j,j\right)  a_{N}}\left(  \frac{x}{a_{N}}\right)  _{+}^{j-1}\left(
1-\frac{x}{a_{N}}\right)  _{+}^{Nk-j-1},
\end{equation}
-- see Theorem~\ref{cxshad}.
 Because $\mathcal{T}_{k}p_{N}\left(  x\right)  =0$
is a polynomial equation it holds for all $x\neq0,a_{N}$. The piecewise
polynomial $p_{N}$ has coefficients which are symmetric in $a_{1}
,\ldots,a_{N-1}$ -- see equation~(\ref{kparfrac}).
Hence the differential equation 
is symmetric in $\left(  a_{1},\ldots,a_{N}\right)$ 
and each piece
\begin{equation}
p_{i}\left(  x\right)  :=\sum\limits_{j=1}^{k}\beta_{ij}\dfrac{1}{B\left(
Nk-j,j\right)  a_{i}}\left(  \frac{x}{a_{i}}\right)  _{+}^{j-1}\left(
1-\frac{x}{a_{i}}\right)  _{+}^{Nk-j-1}%
\end{equation}
satisfies the differential equation for $x\neq0,a_{i}$. 
The density is the sum
$\sum_{i=1}^{N}p_{i}$ thus $\mathcal{T}_{k}f\left(  x\right)  =0$ at each
$x\notin\left\{  a_{1},\ldots,a_{N}\right\}  $. The density $f_{c}\left(
x\right)  $ of \linebreak $\mathcal{D}\left(  
a_{1}+c,\ldots,a_{N}+c;k,\ldots,k\right)  $
equals the translate $f\left(  x-c\right)  $ and the differential operator
$\mathcal{T}_{k}$ has a corresponding translation property and thus the
restriction $a_{1}>0$ can be removed.
\end{proof}
\medskip

\begin{corollary}
If $k>0$ then the density $f\left(  x\right)  $ of $\mathcal{D}\left(
a_{1},\ldots,a_{N};k,\ldots,k\right)  $ satisfies the linear differential
equation $\mathcal{T}_{k}f\left(  x\right)  =0$ at all $x\notin\left\{
a_{1},\ldots,a_{N}\right\}  $.
\end{corollary}

\begin{proof}
Suppose $a_{M}<x_{1}<x_{2}<a_{M+1}$. The probability $\Pr\left\{
x_{1}<X<x_{2}\right\}  $ is given by a definite integral with respect to an
integrand which is analytic for $\operatorname{Re}k>0$ over a polytope in
$\mathbb{T}_{N-1}$ whose vertices are independent of $k$ and analytic in
$x_{1},x_{2}$ -- see Remark \ref{polyt}. Thus the distribution function
$F\left(  x\right)  $ at $x$ is analytic for $\operatorname{Re}k>0$ and
extends to an analytic function in $x$ for 
$x_{1}<\operatorname{Re} x <x_{2},\left\vert \operatorname{Im}x\right\vert <\varepsilon$ for some $\varepsilon>0$. 
The differential equation $\mathcal{T}_{k}\frac{\partial
}{\partial x}F\left(  x\right)  =0$ is satisfied for each $k=1,2,3,\ldots$,
this is an analytic relation and extends to all $\operatorname{Re}k>0$ by
Carlson's theorem (see Henrici \cite[vol.2,p.334]{H}).
\end{proof}
\medskip

We can now assert the validity of the equation for $k=\frac{1}{2},\frac{3}%
{2},\ldots$ which applies to real shadows or the repeated eigenvalue case
(each is repeated 3 times, or 5 times, etc.). It is not clear what happens if
just one eigenvalue is repeated, note that the main result used symmetric
functions of $\frac{1}{a_{i}-a_{N}}$. It is plausible that the equation
applies in intervals adjacent to simple (non-repeated) eigenvalues.

The case $k=\frac{1}{2}$ is of special interest since it applies to the real
shadow when the eigenvalues are pairwise distinct. The equation $\mathcal{T}%
_{1/2}f\left(  x\right)  =0$ is%
\begin{equation}
P_{N}\left(  x\right)  \partial^{N-1}f\left(  x\right)  +\sum_{j=1}%
^{N-1}\left(  -1\right)  ^{j}\frac{N-j}{N}\frac{\left(  -N/2\right)  _{j}}%
{j!}\partial^{j}P_{N}\left(  x\right)  \partial^{N-1-j}f\left(  x\right)  =0.
\end{equation}
When $N$ is even then the terms $\partial^{m}f\left(  x\right)  $ for $0\leq
m\leq\frac{N}{2}-2$ drop out, that is any polynomial of degree $\frac{N}{2}-2$
satisfies the equation. This property will be made precise in the next section.

The \textit{indicial equation} is important because it provides information
about the solutions in neighborhoods of the knots, that is, the solutions have
the form%
\begin{equation}
\sum_{n=0}^{\infty}\gamma_{n}\left(  x-a_{j}\right)  ^{n+c},\sum_{n=0}%
^{\infty}\gamma_{n}\left(  a_{j}-x\right)  ^{n+c},
\end{equation}
(depending on whether the solution is valid for $x>a_{j}$ or $x<a_{j}$) where
$c$ is a solution of the indicial equation: this comes from the coefficient of
the lowest power in $\mathcal{T}_{k}\left(  a_{N}-x\right)  ^{c}$ from
equation (\ref{Tkxc}), namely%
\begin{equation}
e_{N-1}A\left(  N,k,1,c\right)  =\left(  -1\right)  ^{N-1}e_{N-1}\left(
-c\right)  _{N-2}\left(  c+1-\left(  N-1\right)  k\right)  =0.
\end{equation}
The solutions, called \textit{critical exponents}, are $c=0,1,\ldots
,N-3,\left(  N-1\right)  k-1$. In the real shadow situation with $k=\frac
{1}{2}$ we see there are two different types: when $N=2m+1$ the critical
exponent $c=m-1$ is repeated which leads to a logarithmic solution:
$\sum_{n=0}^{\infty}$ $\gamma_{n}\left(  x-a_{j}\right)  ^{m-1+n}$ and
$\log\left\vert x-a_{j}\right\vert \sum_{n=0}^{\infty}$ $\gamma_{n}^{\prime
}\left(  x-a_{j}\right)  ^{m-1+n}$. This actually occurs, as will be shown in
the sequel.

\section{The real shadow}\label{sec:real-shadow}

We will use \textquotedblleft heuristic extrapolation\textquotedblright\ to
postulate a set of formulas for the real shadow (\ref{realshad})
-- the density of $\mathcal{D}\left(
a_{1},\ldots,a_{N};\frac{1}{2},\ldots,\frac{1}{2}\right)$.
 In the notation
of Theorem~\ref{shadeq1} there is a set of functions $p_{j}\left(  x\right)
$, with a  symmetry property, such that the density $f\left(  x\right)
=\sum_{j=m}^{N}p_{j}\left(  x\right)  $ in the interval $\left(  a_{m-1}%
,a_{m}\right)  $. It is straightforward to do this in the top interval
$\left(  a_{N-1},a_{N}\right)$ but the expression involves square roots of
quantities that become negative for $x<a_{N-1}$. The idea is to adopt certain
branches of the complex square roots which give plausible results and then to
prove the validity of the postulated formulas. This will be done by using
complex contour integration to verify the known moment generating function.

We begin by pointing out that the expression for the density in $\left(
a_{N-1},a_{N}\right)  $ found in Proposition \ref{topF} is a multiple infinite
series which diverges for $\left\vert x-a_{N}\right\vert >a_{N}-a_{N-1}$, not
an easy expression to evaluate. We can replace it by a one-variable (definite)
integral when $k=\frac{1}{2}$. Suppose the series $g\left(  r\right)
=\sum_{n=0}^{\infty}\gamma_{n}r^{n}$ converges for $\left\vert r\right\vert
\leq1$ then%
\begin{equation}
\frac{1}{B\left(  \frac{1}{2},\frac{N}{2}-1\right)  }\int_{0}^{1}\sum
_{n=0}^{\infty}\gamma_{n}t^{n}t^{-1/2}\left(  1-t\right)  ^{N/2-2}%
dt=\sum_{n=0}^{\infty}\frac{\left(  \frac{1}{2}\right)  _{n}}{\left(
\frac{N-1}{2}\right)  _{n}}\gamma_{n}.
\end{equation}
Apply this to
\begin{equation}
g\left(  r\right)  =%
{\textstyle\prod_{j=1}^{N-1}}
\left(  1-\frac{a_{N}-x}{a_{N}-a_{j}}r\right)  ^{-\frac{1}{2}}=\sum_{\alpha
\in\mathbb{N}_{0}^{N-1}}\prod_{j=1}^{N-1}\frac{\left(  k_{i}\right)
_{\alpha_{j}}}{\alpha_{j}!}\left(  \frac{a_{N}-x}{a_{N}-a_{j}}\right)
^{a_{j}}r^{\left\vert \alpha\right\vert }%
\end{equation}
and use the formula for the density from Proposition \ref{topF} 
and act with $-\frac{\partial}{\partial x}$ on $1-F\left(  x\right)$
 to obtain the density for $x \in (a_{N-1},a_{N})$,
\begin{equation}
\begin{split}
f\left(  x\right)  = &\frac{N-2}{2\pi\left(  a_{N}-x\right)  }\prod_{j=1}%
^{N-1}\left(  \frac{a_{N}-x}{a_{N}-a_{j}}\right)  ^{\frac{1}{2}} \times \\
& \times \int_{0}%
^{1}\prod_{j=1}^{N-1}\left(  1-\frac{a_{N}-x}{a_{N}-a_{j}}t\right)
^{-\frac{1}{2}}t^{-1/2}\left(  1-t\right)  ^{N/2-2}dt.
\end{split}
\end{equation}
Note that $B\left(  \frac{1}{2},\frac{N}{2}-1\right)  B\left(  \frac{1}{2}%
,\frac{N-1}{2}\right)  =\frac{\Gamma\left(  1/2\right)  ^{2}\Gamma\left(
N/2-1\right)  }{\Gamma\left(  N/2\right)  }=\frac{\pi}{N/2-1}$. 
Make the change of variable $s=a_{N}-t\left(  a_{N}-x\right)  $, then%
\begin{equation}
f\left(  x\right)  =\frac{N-2}{2\pi}
\int_{x}^{a_{N}}\left(  a_{N}-s\right)
^{-\frac{1}{2}}\prod_{j=1}^{N-1}\left(  s-a_{j}\right)  ^{-\frac{1}{2}}\left(
s-x\right)  ^{\frac{N}{2}-2}ds.
\end{equation}
Suppose we want to interpret this integral for $a_{N-2}<x<a_{N-1}$ then we
must pick a branch of $\left(  s-a_{N-1}\right)  ^{-\frac{1}{2}}$, that is we
need to choose the sign in $\left(  s-a_{N-1}\right)  ^{-\frac{1}{2}}%
=\pm\mathrm{i}\left(  a_{N-1}-s\right)  ^{-\frac{1}{2}}$ , where
$\mathrm{i}=\sqrt{-1}$. Denote the integral by $f_{N}\left(  x\right)  $.
Using the symmetry heuristics we define
\begin{equation}
f_{N-1}\left(  x\right)  =\frac{N-2}{2\pi}\int_{x}^{a_{N-1}}\left(
a_{N-1}-s\right)  ^{-\frac{1}{2}}\prod_{j=1,j\neq N-1}^{N}\left(
s-a_{j}\right)  ^{-\frac{1}{2}}\left(  s-x\right)  ^{\frac{N}{2}-2}ds,
\end{equation}
now we need to pick a branch for $\left(  s-a_{N}\right)  ^{-\frac{1}{2}}$ for
$s<a_{N}$. The requirement that $f_{N}\left(  x\right)  +f_{N-1}\left(
x\right)  $ be real for $a_{N-2}<x<a_{N-1}$ motivates the following:

\begin{enumerate}
\item For $0\leq j<N$ and $a_{1}<x\leq a_{N-j}$ let%
\begin{equation}
\begin{split}
f_{N-j}\left(  x\right)  = & \frac{N-2}{2\pi}\mathrm{i}^{j}\int_{x}^{a_{N-j}%
}\bigg( \prod\limits_{m=0}^{j} \left(  a_{N-m}-s\right)  ^{-\frac{1}{2}} 
\times
\\%
& \times \prod\limits_{m=j+1}^{N-1}\left(  s-a_{N-m}\right)  
^{-\frac{1}{2}}\left(s-x\right)  ^{\frac{N}{2}-2}\bigg)ds,
\end{split}
\end{equation}

\item for $0\leq j\leq N-2$ and $a_{N-j-1}\leq x<a_{N-j}$ the density is%
\begin{equation}
f\left(  x\right)  =\sum_{m=0}^{j}f_{N-m}\left(  x\right)  ,
\end{equation}

\item if $s<a_{m}$ then $\left(  s-a_{m}\right)  ^{-\frac{1}{2}}%
=-\mathrm{i}\left(  a_{m}-s\right)  ^{-\frac{1}{2}}$ for $2\leq m\leq N-1.$
\end{enumerate}

Suppose $a_{1}<x<a_{N-j-1}$ for some $j\geq 0$.
As a consequence we obtain then
\begin{equation}
\begin{split}
& f_{N-j}\left(  x\right)  +f_{N-j-1}\left(  x\right)  = \\
& \frac{N-2}{2\pi
}\mathrm{i}^{j}\int_{a_{N-j-1}}^{a_{N-j}}\prod\limits_{m=0}^{j}\left(
a_{N-m}-s\right)  ^{-\frac{1}{2}}\prod\limits_{m=j+1}^{N-1}\left(
s-a_{N-m}\right)  ^{-\frac{1}{2}}\left(  s-x\right)  ^{\frac{N}{2}-2}ds\\
& +\frac{N-2}{2\pi}\mathrm{i}^{j}\int_{x}^{a_{N-j-1}}\prod\limits_{m=0}%
^{j}\left(  a_{N-m}-s\right)  ^{-\frac{1}{2}}\prod\limits_{m=j+1}^{N-1}\left(
s-a_{N-m}\right)  ^{-\frac{1}{2}}\left(  s-x\right)  ^{\frac{N}{2}-2}ds\\
& +\frac{N-2}{2\pi}\mathrm{i}^{j+1}\int_{x}^{a_{N-j-1}}\prod\limits_{m=0}%
^{j+1}\left(  a_{N-m}-s\right)  ^{-\frac{1}{2}}\prod\limits_{m=j+2}%
^{N-1}\left(  s-a_{N-m}\right)  ^{-\frac{1}{2}}\left(  s-x\right)  ^{\frac
{N}{2}-2}ds.
\end{split}
\end{equation}
Due to equation (\ref{dirinteg})
 the factor $\left(  s-a_{N-j-1}\right)  ^{-\frac{1}{2}}$ 
in the second
integral is replaced by $-\mathrm{i}\left(  a_{N-j-1}-s\right)  ^{-\frac{1}%
{2}}$. Therefore the second and third integrals cancel out 
as  $\left(-\mathrm{i}\right)  \mathrm{i}^{j}+\mathrm{i}^{j+1}=0$.
Hence there are two different types of expressions 
for the density, depending on whether
$a_{N-2M}<x<a_{N-2M+1}$ or $a_{N-2M-1}<x<a_{N-2M}$. For $0\leq j\leq
\left\lfloor \frac{N-2}{2}\right\rfloor $ let%
\begin{equation}
g_{j}\left(  s\right)  :=\prod\limits_{m=0}^{2j}\left(  a_{N-m}-s\right)
^{-\frac{1}{2}}\prod\limits_{m=2j+1}^{N-1}\left(  s-a_{N-m}\right)
^{-\frac{1}{2}},
\end{equation}
then for $a_{N-2M}\leq x<a_{N-2M+1}$ (with $1\leq M\leq\frac{N-1}{2}$)
\begin{equation}
f\left(  x\right)  =\frac{N-2}{2\pi}\sum_{j=0}^{M-1}\left(  -1\right)
^{j}\int_{a_{N-2j-1}}^{a_{N-2j}}g_{j}\left(  s\right)  \left(  s-x\right)
^{\frac{N}{2}-2}ds,\label{evenF}%
\end{equation}
and for $a_{N-2M-1}\leq x<a_{N-2M}$ (with $0\leq M\leq\frac{N-2}{2}$)%
\begin{align}
f\left(  x\right)    & =\frac{N-2}{2\pi}\sum_{j=0}^{M-1}\left(  -1\right)
^{j}\int_{a_{N-2j-1}}^{a_{N-2j}}g_{j}\left(  s\right)  \left(  s-x\right)
^{\frac{N}{2}-2}ds\label{oddF}\\
& +\left(  -1\right)  ^{M}\frac{N-2}{2\pi}\int_{x}^{a_{N-2M}}g_{M}\left(
s\right)  \left(  s-x\right)  ^{\frac{N}{2}-2}ds.\nonumber
\end{align}
An important consequence of this formulation is that for even $N$ the density
is a polynomial of degree $\frac{N}{2}-2$ on the \textit{even} intervals
$\left(  a_{N-2M},a_{N-2M+1}\right)$,
which means that the parity by counting
intervals from the top down is even,
 so that $\left(  a_{N-1},a_{N}\right)  $ is \#1.

\medskip 
Now we are in position to formulate the main result of this work.
\begin{theorem}
\label{th:main-theorem}
For $N>2$ the formulas (\ref{evenF}) and (\ref{oddF}) give the real shadow 
of a real symmetric matrix with spectrum $\{a_i\}_{i=1}^N$ 
-- the density of 
$\mathcal{P}^{\mathbb{R}}_{\mathrm{diag}(a_{1},\ldots,a_{N})} = \mathcal{D}\left(  a_{1},\ldots,a_{N};\frac{1}{2},\ldots,\frac
{1}{2}\right)  $.
\end{theorem}

\smallskip 

We prove the validity of the above theorem by showing that%
\begin{equation}
\int_{a_{1}}^{a_{N}}\left(  1-r\left(  x-a_{1}\right)  \right)  ^{-\frac{N}%
{2}}f\left(  x\right)  dx=\prod_{j=2}^{N}\left(  1-r\left(  a_{j}%
-a_{1}\right)  \right)  ^{-\frac{1}{2}},\left\vert r\right\vert <\frac
{1}{a_{N}-a_{1}},
\end{equation}
this is the moment generating function, see Lemma \ref{mgenF}. Start by
expressing $\int_{a_{1}}^{a_{N}}\left(  x-a_{1}\right)  ^{n}f\left(  x\right)
dx$ as a sum of integrals, for $n=0,1,2,\ldots$. The contribution of an
\textquotedblleft even\textquotedblright\ interval $a_{N-2M}\leq x\leq
a_{N-2M+1}$ to the moment is%
\begin{equation}
\frac{N-2}{2\pi}\sum_{j=0}^{M-1}\left(  -1\right)  ^{j}\int_{a_{N-2M}%
}^{a_{N-2M+1}}\left(  x-a_{1}\right)  ^{n}dx\int_{a_{N-2j-1}}^{a_{N-2j}}%
g_{j}\left(  s\right)  \left(  s-x\right)  ^{\frac{N}{2}-2}ds,
\end{equation}
and the contribution of an \textquotedblleft odd\textquotedblright\ interval
$a_{N-2M-1}\leq x\leq a_{N-2M}$ is%
\begin{equation}
\begin{split}
&  \frac{N-2}{2\pi}\sum_{j=0}^{M-1}\left(  -1\right)  ^{j}\int_{a_{N-2M-1}%
}^{a_{N-2M}}\left(  x-a_{1}\right)  ^{n}dx\int_{a_{N-2j-1}}^{a_{N-2j}}%
g_{j}\left(  s\right)  \left(  s-x\right)  ^{\frac{N}{2}-2}ds\\
&  +\left(  -1\right)  ^{M}\frac{N-2}{2\pi}\int_{a_{N-2M-1}}^{a_{N-2M}}\left(
x-a_{1}\right)  ^{n}dx\int_{x}^{a_{N-2M}}g_{M}\left(  s\right)  \left(
s-x\right)  ^{\frac{N}{2}-2}ds.
\end{split}
\end{equation}
The term $g_{j}\left(  s\right)  \left(  s-x\right)  ^{\frac{N}{2}-2}$ appears
in the intervals $a_{N-2M}\leq x\leq a_{N-2M+1}$ for $M\geq j+1$ and in
$a_{N-2M-1}\leq x\leq a_{N-2M}$ for $M\geq j$. Collect these terms:%
\begin{equation}
\begin{split}
&  
\!\!\!
\frac{N-2}{2\pi}\left(  -1\right)^{j}
\!\!\!
\int_{a_{N-2j-1}}^{a_{N-2j}}
\!\!\!\!\!\!\!\!\!\!\!\!
g_{2j}
\!
\left(  s\right)
\!
\left\{ \sum_{M=j+1}^{\left\lfloor \frac{N-1}%
{2}\right\rfloor }
\!
\int_{a_{N-2M}}^{a_{N-2M+1}}
\!\!\!\!
+
 \!\!
\sum_{M=j+1}^{\left\lfloor
\frac{N-2}{2}\right\rfloor}
\!
\int_{a_{N-2M-1}}^{a_{N-2M}}\right\}
\!\!
  \left(x-a_{1}\right)^{n}
\!
\left(s-x\right)  ^{\frac{N}{2}-2}dxds\\
&  +\frac{N-2}{2\pi}\left(  -1\right)  ^{j}\int_{a_{N-2j-1}}^{a_{N-2j}}\left(
x-a_{1}\right)  ^{n}dx\int_{x}^{a_{N-2j}}g_{j}\left(  s\right)  \left(
s-x\right)  ^{\frac{N}{2}-2}ds.
\end{split}
\end{equation}
The terms in the first line add up to just one interval of integration
$a_{1}\leq x\leq a_{N-2j-1}$. In the second line reverse the order of
integration (note the region for the double integral is $a_{N-2j-1}\leq x\leq
s\leq a_{N-2j}$) to obtain%
\begin{equation}
\frac{N-2}{2\pi}\left(  -1\right)  ^{j}\int_{a_{N-2j-1}}^{a_{N-2j}}%
g_{j}\left(  s\right)  ds\int_{a_{N-2j-1}}^{s}\left(  x-a_{1}\right)
^{n}\left(  s-x\right)  ^{\frac{N}{2}-2}dx.
\end{equation}
The terms with $g_{j}$ add up to
\begin{equation}
\begin{split}
&  \frac{N-2}{2\pi}\left(  -1\right)  ^{j}\int_{a_{N-2j-1}}^{a_{N-2j}}%
g_{j}\left(  s\right)  ds\int_{a_{1}}^{s}\left(  x-a_{1}\right)  ^{n}\left(
s-x\right)  ^{\frac{N}{2}-2}dx\\
&  =\frac{\left(  -1\right)  ^{j}}{\pi}\left(  \frac{N-2}{2}\right)  B\left(
\frac{N}{2}-1,n+1\right)  \int_{a_{N-2j-1}}^{a_{N-2j}}g_{j}\left(  s\right)
\left(  s-a_{1}\right)  ^{\frac{N}{2}+n-1}ds,
\end{split}
\end{equation}
from the Beta integral $\int_{a}^{b}\left(  b-x\right)  ^{\alpha-1}\left(
x-a\right)  ^{\beta-1}dx=\left(  b-a\right)  ^{\alpha+\beta-1}B\left(
\alpha,\beta\right)$ with $a=a_{1},b=s,\alpha=\frac{N}{2}-1,\beta=n+1$.
Furthermore $\left(  \frac{N-2}{2}\right)  B\left(  \frac{N}{2}-1,n+1\right)
=\dfrac{n!}{\left(  \frac{N}{2}\right)  _{n}}$. 
Therefore we have
\begin{equation}
\int_{a_{1}}^{a_{N}}\left(  x-a_{1}\right)  ^{n}f\left(  x\right)  dx=\frac
{1}{\pi}\frac{n!}{\left(  \frac{N}{2}\right)  _{n}}\sum_{j=0}^{\left\lfloor
\frac{N-2}{2}\right\rfloor }\left(  -1\right)  ^{j}\int_{a_{N-2j-1}}%
^{a_{N-2j}}g_{j}\left(  s\right)  \left(  s-a_{1}\right)  ^{\frac{N}{2}%
+n-1}ds,
\end{equation}
and%
\begin{gather}
\nonumber
\int_{a_{1}}^{a_{N}}\left(  1-r\left(  x-a_{1}\right)  \right)  ^{-\frac{N}%
{2}}f\left(  x\right)  dx=\sum_{n=0}^{\infty}\frac{\left(  \frac{N}{2}\right)
_{n}}{n!}r^{n}\int_{a_{1}}^{a_{N}}\left(  x-a_{1}\right)  ^{n}f\left(
x\right)  dx\\
=\frac{1}{\pi}\sum_{j=0}^{\left\lfloor \frac{N-2}{2}\right\rfloor }\left(
-1\right)  ^{j}\int_{a_{N-2j-1}}^{a_{N-2j}}g_{j}\left(  s\right)  \left(
s-a_{1}\right)  ^{\frac{N}{2}-1}\left(  1-r\left(  s-a_{1}\right)  \right)
^{-1}ds,
\end{gather}
where the infinite sum converges for 
$\left\vert r\right\vert <\frac{1}{a_{N}-a_{1}}$.
 We will evaluate the integral by residue
 calculus applied to the analytic function
\begin{equation}\label{eqn:func-G}
G\left(  z\right)  :=\prod\limits_{j=1}^{N}\left(  z-a_{j}\right)  ^{-\frac
{1}{2}}
\left(  z-a_{1}\right) ^{\frac{N}{2}-1}\left(  1-r\left(
z-a_{1}\right)  \right)  ^{-1}
\end{equation}
for fixed small $r>0$
with suitable determination of the square roots.
For real $a,b$ with $a<b$ consider the analytic function $\left(  z-a\right)
^{-\frac{1}{2}}\left(  z-b\right)  ^{-\frac{1}{2}}$ defined on $\mathbb{C}%
\backslash\left[  a,b\right]  $, that is, the complex plane with the interval
$\left[  a,b\right]  $ removed. Set $z=a+r_{1}e^{\mathrm{i}\theta_{1}}%
=b+r_{2}e^{\mathrm{i}\theta_{2}}$, $r_{1},r_{2}>0$ and $\theta_{1}=\theta
_{2}=0$ for $z$ real and $z>b$, then let%
\begin{equation}
\left(  z-a\right)  ^{-\frac{1}{2}}\left(  z-b\right)  ^{-\frac{1}{2}%
}:=\left(  r_{1}r_{2}\right)  ^{-\frac{1}{2}}e^{-\mathrm{i}\left(  \theta
_{1}+\theta_{2}\right)  /2}%
\end{equation}
and let $\theta_{1},\theta_{2}$ vary continuously (from $0$) to determine the
values in the rest of the domain. This is well-defined: suppose $z$ is real
and $z<a$; approaching $z$ from the upper half-plane $\theta_{1},\theta_{2}$
change from $0$ to $\pi$ and $e^{-\mathrm{i}\left(  \theta_{1}+\theta
_{2}\right)  /2}$ changes from $1$ to $e^{-\mathrm{i}\pi}=-1$, and approaching
$z$ from the lower half-plane $\theta_{1},\theta_{2}$ change from $0$ to
$-\pi$ and $e^{-\mathrm{i}\left(  \theta_{1}+\theta_{2}\right)  /2}$ changes
from $1$ to $e^{\mathrm{i}\pi}=-1$.

\begin{lemma}
Suppose $h$ is analytic in a complex neighborhood of $\left[  a,b\right]  $
and $\gamma_{\varepsilon}$ is a closed contour oriented clockwise (negatively)
made up of the segments $\left\{  x+\mathrm{i}\varepsilon:a\leq x\leq
b\right\}  $, $\left\{  x-\mathrm{i}\varepsilon:a\leq x\leq b\right\}  $ and
semicircles $\left\{  a+\varepsilon e^{\mathrm{i}\theta}:\frac{\pi}{2}%
\leq\theta\leq\frac{3\pi}{2}\right\}  $, $\left\{  b+\varepsilon
e^{\mathrm{i}\theta}:-\frac{\pi}{2}\leq\theta\leq\frac{\pi}{2}\right\}  $ (for
sufficiently small $\varepsilon>0$) then%
\begin{equation}
\lim_{\varepsilon\rightarrow0_{+}}\oint_{\gamma_\epsilon} h\left(  z\right)  
\left(  z-a\right)
^{-\frac{1}{2}}\left(  z-b\right)  ^{-\frac{1}{2}}dz=-2\mathrm{i}\int_{a}%
^{b}h\left(  x\right)  \left(  \left(  b-x\right)  \left(  a-x\right)
\right)  ^{-\frac{1}{2}}dx.
\end{equation}

\end{lemma}

\begin{proof}
On the semicircles the integrand is bounded by $M\varepsilon^{-\frac{1}{2}}$
for some $M<\infty$ and the length of the arc is $\pi\varepsilon$ so this part
of the integral tends to zero as $\varepsilon\rightarrow0_{+}$. Along
$\left\{  z=x+\mathrm{i}\varepsilon:a\leq x\leq b\right\}  $ the arguments are
$\theta_{1}\approx\pi$ and $\theta_{2}\approx0$ so $\left(  z-a\right)
^{-\frac{1}{2}}\left(  z-b\right)  ^{-\frac{1}{2}}\approx e^{-\mathrm{i}\pi
/2}\left(  r_{1}r_{2}\right)  ^{-\frac{1}{2}}$ and this part of the integral
$\approx-\mathrm{i}\int_{a}^{b}h\left(  x+\mathrm{i}\varepsilon\right)
\left(  \left(  b-x\right)  \left(  a-x\right)  \right)  ^{-\frac{1}{2}}dx$.
Along $\left\{  z=x-\mathrm{i}\varepsilon:a\leq x\leq b\right\}  $ the
arguments are $\theta_{1}\approx-\pi$ and $\theta_{2}\approx0$ so $\left(
z-a\right)  ^{-\frac{1}{2}}\left(  z-b\right)  ^{-\frac{1}{2}}\approx
e^{\mathrm{i}\pi/2}\left(  r_{1}r_{2}\right)  ^{-\frac{1}{2}}$ and this part
of the integral $\approx\mathrm{i}\int_{b}^{a}h\left(  x-\mathrm{i}%
\varepsilon\right)  \left(  \left(  b-x\right)  \left(  a-x\right)  \right)
^{-\frac{1}{2}}dx$. Adding the two pieces and letting $\varepsilon
\rightarrow0_{+}$ proves the claim.
\end{proof}

\medskip 
Now fix $r>0$ with $\frac{1}{r}>\max\left(  \left\vert a_{1}\right\vert
,\left\vert a_{N}\right\vert ,a_{N}-a_{1}\right)  $. Define a positively
oriented closed contour $\Gamma$ consisting of a large circle $\gamma=\left\{
z=\operatorname{Re}^{\mathrm{i}\theta}:0\leq\theta\leq2\pi\right\}  $ with
$R>\frac{1}{r}$ and $\left\{  \gamma_{j,\varepsilon}:0\leq j\leq\left\lfloor
\frac{N-2}{2}\right\rfloor \right\}  $ where $\gamma_{j,\varepsilon}$ is a
closed negatively oriented contour around the interval $\left[  a_{N-2j-1}%
,a_{N-2j}\right]  $ as in the Lemma, with $\varepsilon>0$ sufficiently small
so that the contours do not intersect --
see Fig.~\ref{fig:poles}. The function $G$ is meromorphic on
$\mathbb{C}\backslash\cup_{j=0}^{\left\lfloor \left(  N-2\right)
/2\right\rfloor }\left[  a_{N-2j-1},a_{N-2j}\right]  $ and has one simple pole
at $z=a_{1}+\frac{1}{r}$. By the (generalized) residue theorem%
\begin{equation}\label{eqn:complex-integral}
\frac{1}{2\pi\mathrm{i}}\oint\limits_{\Gamma}G\left(  z\right)
dz=\mathrm{res}_{z=a_{1}+\frac{1}{r}}G\left(  z\right)  .
\end{equation}
Using the determinations of roots described above let $z=a_{j}+r_{j}%
e^{\mathrm{i}\theta_{j}}$ for $1\leq j\leq N$ with $r_{j}>0$. For large
$\left\vert z\right\vert $ we see $\left\vert G\left(  z\right)  \right\vert
<M\left\vert z\right\vert ^{-2}$ so the integral around $\gamma$ (circle with
radius $R)$ tends to zero as $R\rightarrow\infty$. Consider $N$ even or odd 
separately.

\begin{figure}[]
    \centering
    \subfloat[]{\includegraphics[width=0.48\linewidth]{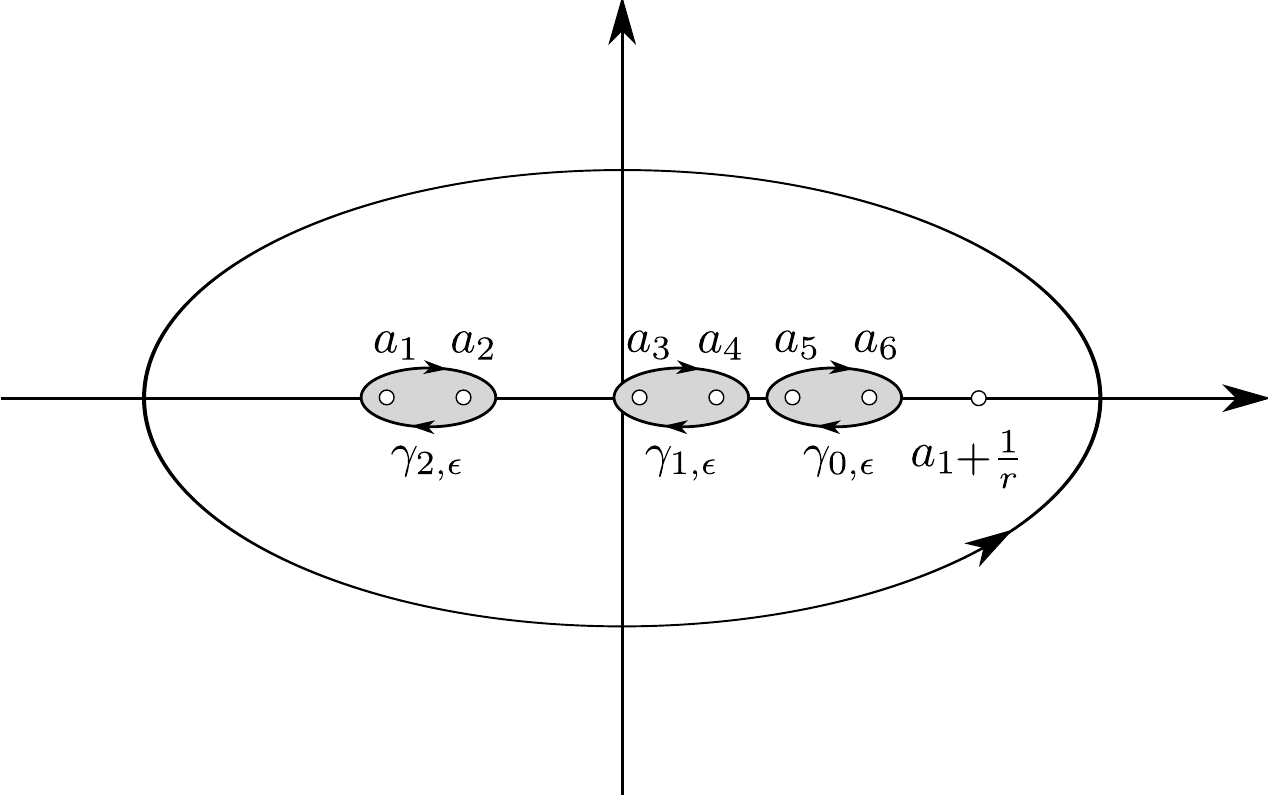}}
    \label{fig:int-even}
    \hspace{0.02\linewidth}
    \subfloat[]{\includegraphics[width=0.48\linewidth]{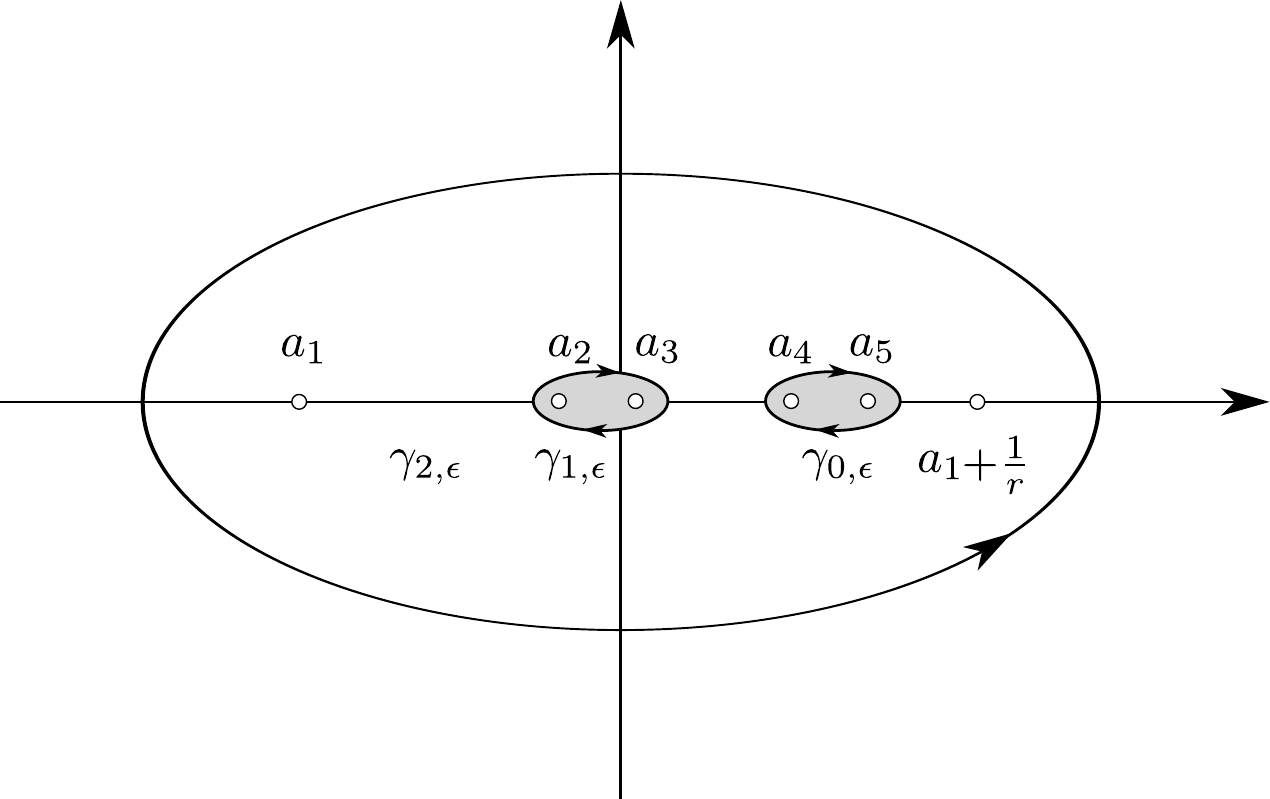}}
    \label{fig:int-odd}
    \caption{Visualization of the integration of the function $G(z)$ defined 
    in~\eqref{eqn:func-G}. Panel (a) even $N$ (here $N=6$), panel (b) odd $N$ 
    (here $N=5$).}
    \label{fig:poles}
\end{figure}
\subsection{Case of odd $N$:}
The interval with the lowest index is $\left[  a_{2},a_{3}\right]$
 and the analytic function 
$G\left(  z\right)  =\prod\limits_{j=2}^{N}\left(  z-a_{j}\right)  ^{-\frac
{1}{2}}\left(  z-a_{1}\right)  ^{\frac{N-3}{2}}\left(  1-r\left(
z-a_{1}\right)  \right)  ^{-1}$. Here $\frac{N-3}{2}$ is an integer thus
$\left(  z-a_{1}\right)  ^{\frac{N-3}{2}}$ is entire. Applying the Lemma to
$\gamma_{j,\varepsilon}$ put%
\begin{equation}
h\left(  z\right)  =\prod\limits_{m=0}^{2j-1}\left(  z-a_{N-m}\right)
^{-\frac{1}{2}}\prod\limits_{m=2j+2}^{N-2}\left(  z-a_{N-m}\right)
^{-\frac{1}{2}}\left(  z-a_{1}\right)  ^{\frac{N-3}{2}}\left(  1-r\left(
z-a_{1}\right)  \right)  ^{-1}.
\end{equation}
In this case $\theta_{m}=\pi$ for $N-2j+1\leq m\leq N$ and $\theta_{m}=0$ for
$1\leq m\leq N-2j-2$ so for $a_{N-2j-1}\leq x\leq a_{N-2j}$ we have
\begin{equation}
\begin{split}
h\left(  x\right)  = & e^{-\mathrm{i}\left(  2j\pi\right)  /2}\prod
\limits_{m=0}^{2j-1}\left(  a_{N-m}-x\right)  ^{-\frac{1}{2}} \times \\
& \times\prod
\limits_{m=2j+2}^{N-2}\left(  x-a_{N-m}\right)  ^{-\frac{1}{2}}\left(
x-a_{1}\right)  ^{\frac{N-3}{2}}\left(  1-r\left(  x-a_{1}\right)  \right)
^{-1}%
\end{split}
\end{equation}
and%
\begin{equation}
\begin{split}
&  \lim_{\varepsilon\rightarrow0_{+}}\oint\limits_{\gamma_{j,\varepsilon}%
}G\left(  z\right)  dz
=-2\mathrm{i}\left(  -1\right)  ^{j}\int_{a_{N-2j-1}}^{a_{N-2j}}%
\bigg( \prod\limits_{m=0}^{2j}\left(  a_{N-m}-x\right)  ^{-\frac{1}{2}} \\%
& \prod\limits_{m=2j+1}^{N-2}\left(  x-a_{N-m}\right)  ^{-\frac{1}{2}}\left(
x-a_{1}\right)  ^{\frac{N-3}{2}}\left(  1-r\left(  x-a_{1}\right)  \right)
^{-1} \bigg) dx.
\end{split}
\end{equation}
The residue at $z=a_{1}+\frac{1}{r}$ is straightforward:
\begin{equation}
\begin{split}
\lim_{z\rightarrow a_{1}+\frac{1}{r}}\left(  z-a_{1}-\frac{1}{r}\right)
G\left(  z\right)   &  =-\frac{1}{r}\prod\limits_{j=2}^{N}\left(  \frac{1}%
{r}-\left(  a_{j}-a_{1}\right)  \right)  ^{-\frac{1}{2}}\left(  \frac{1}%
{r}\right)  ^{\frac{N-3}{2}}\\
&  =-\prod\limits_{j=2}^{N}\left(  1-r\left(  a_{j}-a_{1}\right)  \right)
^{-\frac{1}{2}},
\end{split}
\end{equation}
because $\frac{1}{r}>a_{N}$ and the determination of the roots gives positive
values. Thus in the limit as $\varepsilon\rightarrow0_{+},R\rightarrow\infty$
we obtain%
\begin{eqnarray}
\frac{1}{2\pi\mathrm{i}}\sum_{j=0}^{\frac{N-3}{2}}\left(  -1\right)
^{j}\left(  -2\mathrm{i}\right)  \int_{a_{N-2j-1}}^{a_{N-2j}}g_{2j}\left(
x\right)  \left(  x-a_{1}\right)  ^{\frac{N}{2}-1}\left(  1-r\left(
x-a_{1}\right)  \right)  ^{-1}dx=  \nonumber \\
-\prod\limits_{j=2}^{N}\left(  1-r\left(
a_{j}-a_{1}\right)  \right)  ^{-\frac{1}{2}},
\end{eqnarray}
and this is the required result.

\subsection{Case of $N$ even}

The interval with the lowest index is $\left[  a_{1},a_{2}\right]  $ and
the function 
$G\left(  z\right)  =\prod\limits_{j=1}^{N}\left(  z-a_{j}\right)  ^{-\frac
{1}{2}}\left(  z-a_{1}\right)  ^{\frac{N-2}{2}}\left(  1-r\left(
z-a_{1}\right)  \right)  ^{-1}$. Here $\frac{N-2}{2}$ is an integer so
$\left(  z-a_{1}\right)  ^{\frac{N-2}{2}}$ is entire. Applying the Lemma to
$\gamma_{j,\varepsilon}$ put%
\begin{equation}
h\left(  z\right)  =\prod\limits_{m=0}^{2j-1}\left(  z-a_{N-m}\right)
^{-\frac{1}{2}}\prod\limits_{m=2j+2}^{N-1}\left(  z-a_{N-m}\right)
^{-\frac{1}{2}}\left(  z-a_{1}\right)  ^{\frac{N-2}{2}}\left(  1-r\left(
z-a_{1}\right)  \right)  ^{-1}.
\end{equation}
In this case $\theta_{m}=\pi$ for $N-2j+1\leq m\leq N$ and $\theta_{m}=0$ for
$1\leq m\leq N-2j-2$ so for $a_{N-2j-1}\leq x\leq a_{N-2j}$ we have
\begin{equation}
\begin{split}
h\left(  x\right) = & e^{-\mathrm{i}\left(  2j\pi\right)  /2}\prod
\limits_{m=0}^{2j-1}\left(  a_{N-m}-x\right)  ^{-\frac{1}{2}} \times \\
& \times \prod \limits_{m=2j+2}^{N-1}\left(  x-a_{N-m}\right)  
^{-\frac{1}{2}}\left( x-a_{1}\right)  ^{\frac{N-2}{2}}\left(  1-r\left(  
x-a_{1}\right)  \right)^{-1}%
\end{split}
\end{equation}
and%
\begin{equation}
\begin{split}
&  \lim_{\varepsilon\rightarrow0_{+}}\oint\limits_{\gamma_{j,\varepsilon}%
}G\left(  z\right)  dz
=-2\mathrm{i}\left(  -1\right)  ^{j}\int_{a_{N-2j-1}}^{a_{N-2j}}%
\bigg( \prod\limits_{m=0}^{2j}\left(  a_{N-m}-x\right)  ^{-\frac{1}{2}} \times 
\\%
& \times \prod\limits_{m=2j+1}^{N-1}\left(  x-a_{N-m}\right)  
^{-\frac{1}{2}}\left(
x-a_{1}\right)  ^{\frac{N-2}{2}}\left(  1-r\left(  x-a_{1}\right)  \right)
^{-1} \bigg) dx.
\end{split}
\end{equation}
The residue at $z=a_{1}+\frac{1}{r}$ is:
\begin{equation}
\begin{split}
\lim_{z\rightarrow a_{1}+\frac{1}{r}}\left(  z-a_{1}-\frac{1}{r}\right)
G\left(  z\right)   &  =-\frac{1}{r}\prod\limits_{j=1}^{N}\left(  \frac{1}%
{r}-\left(  a_{j}-a_{1}\right)  \right)  ^{-\frac{1}{2}}\left(  \frac{1}%
{r}\right)  ^{\frac{N-2}{2}}\\
&  =-\prod\limits_{j=2}^{N}
\big(  1-r\left(  a_{j}-a_{1}\right) \big)^{-\frac{1}{2}},
\end{split}
\end{equation}
because $\frac{1}{r}>a_{N}$ and the determination of the roots gives positive
values. Thus in the limit as $\varepsilon\rightarrow0_{+},R\rightarrow\infty$
we obtain the final result
\begin{eqnarray}
\frac{1}{2\pi\mathrm{i}}\sum_{j=0}^{\frac{N-2}{2}}\left(  -1\right)
^{j}\left(  -2\mathrm{i}\right)  \int_{a_{N-2j-1}}^{a_{N-2j}}g_{2j}\left(
x\right)  \left(  x-a_{1}\right)  ^{\frac{N}{2}-1}
\big(  1-r\left(x-a_{1}\right)
  \big)^{-1}ds= \nonumber \\
-\prod\limits_{j=2}^{N}\big(  1-r\left(
a_{j}-a_{1}\right)  \big)  ^{-\frac{1}{2}}.
\end{eqnarray}

For distributions supported by bounded intervals the moment generating
function determines the distribution uniquely. Thus we have established the
Theorem~\ref{th:main-theorem}.

\subsection{Examples}

There is a somewhat disguised complete elliptic integral of the first kind
which appears in $N=3,4,5$. For $b_{1}<b_{2}<b_{3}<b_{4}$ let%
\begin{equation}
E\left(  b_{1},b_{2};b_{3},b_{4}\right)  :=\frac{1}{\pi}\int_{b_{3}}^{b_{4}%
}\left\{  \left(  b_{4}-s\right)  \left(  s-b_{3}\right)  \left(
s-b_{2}\right)  \left(  s-b_{1}\right)  \right\}  ^{-\frac{1}{2}}ds.
\end{equation}
There is a hypergeometric formulation (see formula (\ref{N3&2F1}) with
$k=\frac{1}{2}$):%
\begin{equation}
E\left(  b_{1},b_{2};b_{3},b_{4}\right)  :=\frac{1}{\sqrt{\left(  b_{3}%
-b_{1}\right)  \left(  b_{4}-b_{2}\right)  }}~_{2}F_{1}\left(
\genfrac{}{}{0pt}{}{\frac{1}{2},\frac{1}{2}}{1}%
;\dfrac{\left(  b_{4}-b_{3}\right)  \left(  b_{2}-b_{1}\right)  }{\left(
b_{3}-b_{1}\right)  \left(  b_{4}-b_{2}\right)  }\right)  .
\end{equation}
Consider the density for $N=3$ and $a_{1}<a_{2}<x<a_{3}$; by formula
(\ref{oddF})%
\begin{align}
f\left(  x\right)    & =\frac{1}{2\pi}\int_{x}^{a_{3}}\left(  a_{3}-s\right)
^{-\frac{1}{2}}\prod_{j=1}^{2}\left(  s-a_{j}\right)  ^{-\frac{1}{2}}\left(
s-x\right)  ^{-\frac{1}{2}}ds \\
& =\frac{1}{2}E\left(  a_{1},a_{2};x,a_{3}\right)  .\nonumber 
\end{align}
Similarly formula (\ref{evenF}) shows that $f\left(  x\right)  =\frac{1}%
{2}E\left(  a_{1},x;a_{2},a_{3}\right)  $ for $a_{1}<x<a_{2}$.

Suppose $N=4$ and $a_{2}<x\leq a_{3}$ then by formula (\ref{evenF})%
\begin{align}
f\left(  x\right)    & =\frac{1}{\pi}\int_{a_{3}}^{a_{4}}\left(
a_{4}-s\right)  ^{-\frac{1}{2}}\prod\limits_{j=1}^{3}\left(  s-a_{j}\right)
^{-\frac{1}{2}}ds=f\left(  a_{3}\right)  \\
& =E\left(  a_{1},a_{2};a_{3},a_{4}\right) \nonumber
\end{align}
and the density is constant on this interval.

Suppose $N=5$ and $a_{3}\leq x<a_{4}$ then by formula (\ref{evenF})%
\begin{align}
f\left(  x\right)   &  =\frac{3}{2\pi}\int_{a_{4}}^{a_{5}}\left(
a_{5}-s\right)  ^{-\frac{1}{2}}\prod\limits_{j=1}^{4}\left(  s-a_{j}\right)
^{-\frac{1}{2}}\left(  s-x\right)  ^{\frac{1}{2}}ds,\\
f\left(  a_{3}\right)   &  =\frac{3}{2\pi}\int_{a_{4}}^{a_{5}}\left(
a_{5}-s\right)  ^{-\frac{1}{2}}\left(  s-a_{4}\right)  ^{-\frac{1}{2}}%
\prod\limits_{j=1}^{2}\left(  s-a_{j}\right)  ^{-\frac{1}{2}}ds\\\nonumber
&  =\frac{3}{2}E\left(  a_{1},a_{2};a_{4},a_{5}\right),
\end{align}
which is independent of $a_{3}$.

The integrals in the density formula have the form \linebreak 
$\frac{1}{\pi}\int_{a}^{b}h\left(s\right)  \big( \left(  b-s\right)
\left(  s-a\right)  \big)  ^{-\frac{1}{2}}ds$, where $h$ is differentiable
in a neighborhood of $\left[  a,b\right]  $. The technique of Gauss-Chebyshev
quadrature is well suited for the numerical evaluation of the desired 
densities: Set $\phi\left(  t\right)
=\frac{1}{2}\left(  a+b\right)  +\frac{1}{2}\left(  b-a\right)  t$ then the
sums $\frac{1}{n}\sum_{j=0}^{n-1}h\left(  \phi\left(  \cos\frac{\left(
2j+1\right)  \pi}{2n}\right)  \right)  $ converge rapidly to the integral (as
$n\rightarrow\infty$); typically $n=20$ suffices for reasonable accuracy.

Another way of numerical approximation of a real numerical shadow can be done
by direct numerical integration of a formula for a cumulative distribution
function given in~\cite{provost2000distribution}. 

\begin{figure}[!h]
    \centering
    \subfloat[]{\includegraphics{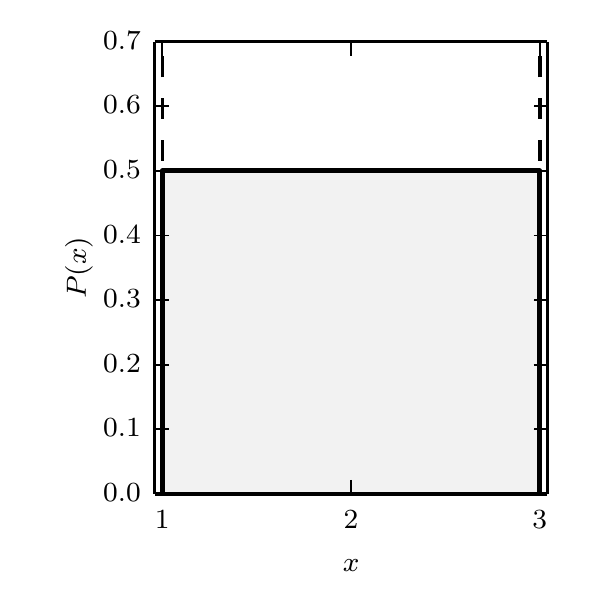}} 
    \subfloat[]{\includegraphics{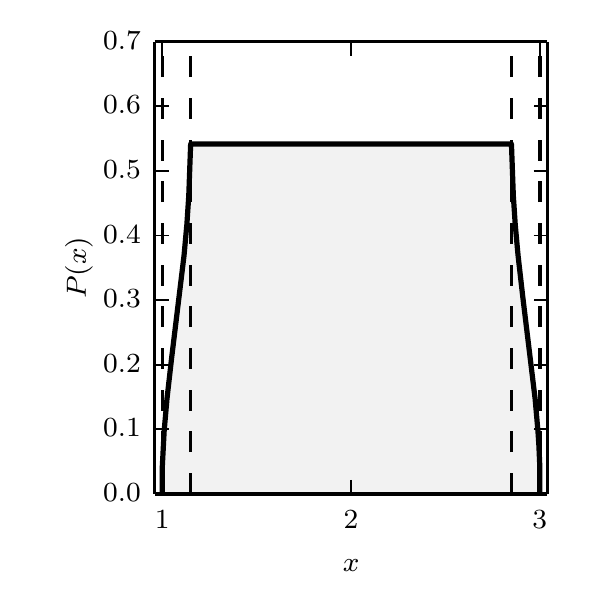}}\\
    \subfloat[]{\includegraphics{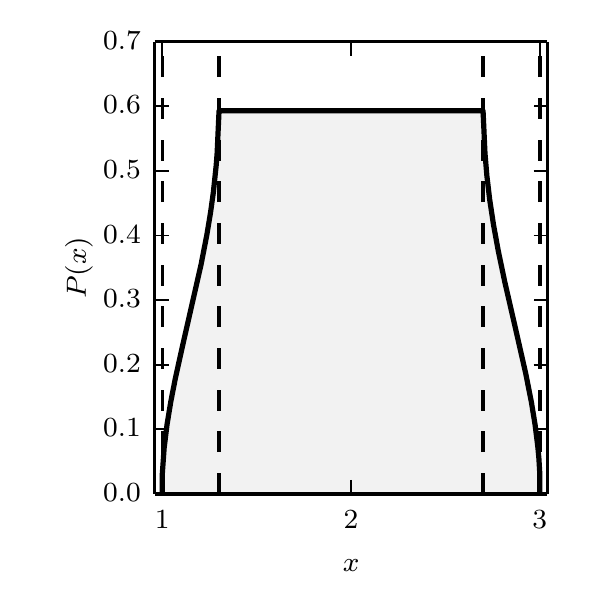}} 
    \subfloat[]{\includegraphics{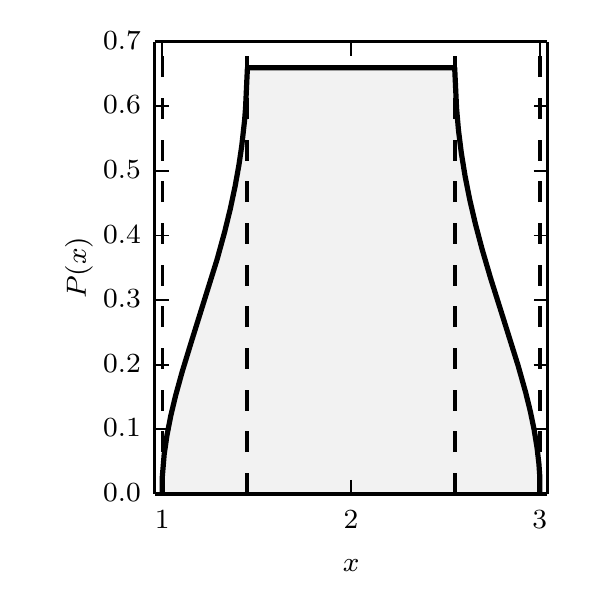}}
    \caption{Real numerical shadow 
$\mathcal{P}^{\mathbb{R}}_A(x)$
of a diagonal matrix   $A=\mathrm{diag}(1,1+\epsilon,3-\epsilon,3)$
of order $N=4$, where  (a) $\epsilon=0$
(degenerated case: the real shadow is equivalent to 
the complex shadow of the reduced matrix, 
${\rm Tr}_2 A=\mathrm{diag}(1,3)$);
 (b) $\epsilon=0.15$; (c) $\epsilon=0.3$ and (d) $\epsilon=0.45$.
Note that for  $x\in(a_2,a_3)$ all distributions are flat.
}
    \label{fig:cplx2realD4}
\end{figure}

\begin{figure}[!h]
    \centering
    \subfloat[]{\includegraphics{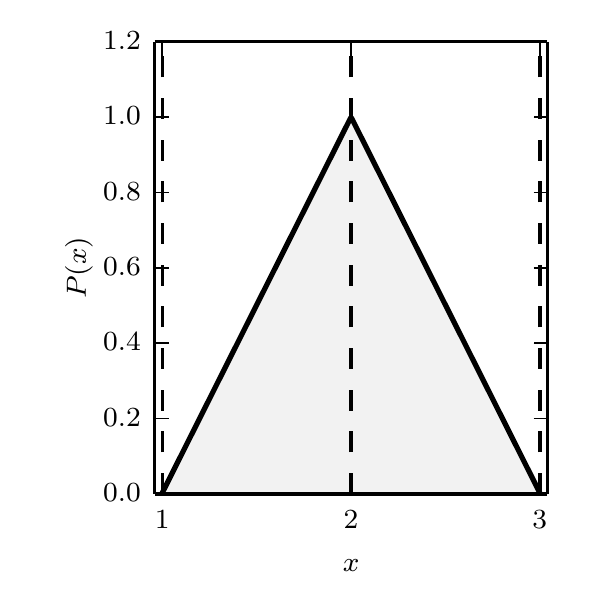}} 
    \subfloat[]{\includegraphics{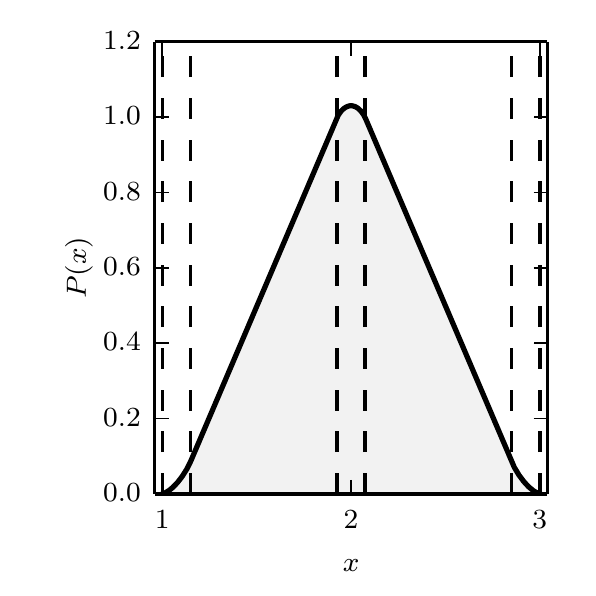}}\\
    \subfloat[]{\includegraphics{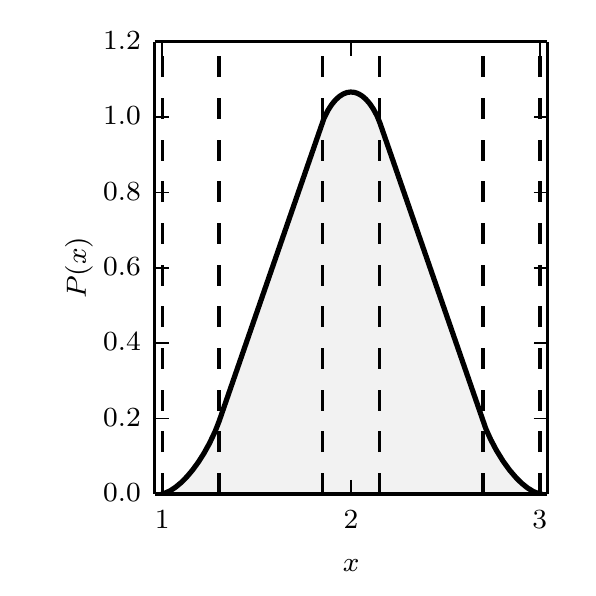}} 
    \subfloat[]{\includegraphics{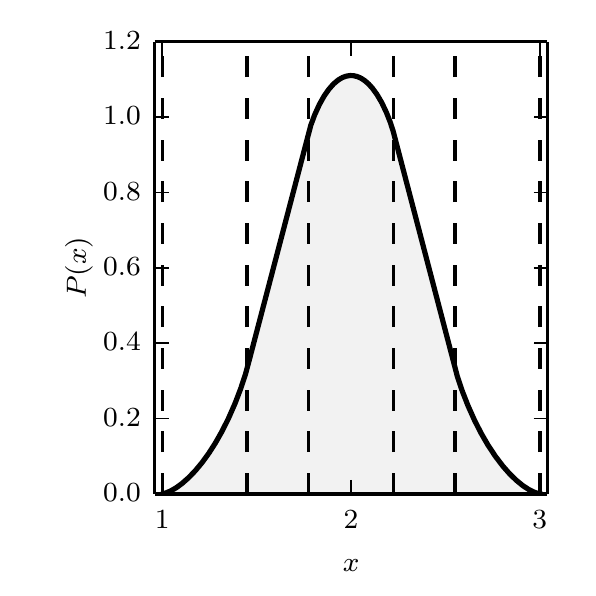}}
    \caption{Real numerical shadow of matrix $A$ of order $N=6$ 
 $A=\mathrm{diag}(1,1+\epsilon,2-\epsilon/2,2+\epsilon/2, 3-\epsilon,3)$,
 where (a) $\epsilon=0$
(degenerated case: the real shadow is equivalent to 
the complex shadow of the reduced matrix, 
${\rm Tr}_2 A=\mathrm{diag}(1,2,3)$);
  (b) $\epsilon=0.15$; (c) $\epsilon=0.3$; 
(d) $\epsilon=0.45$. Note that for  $x\in(a_2,a_3)$ and 
 $x\in(a_4,a_5)$ the distributions are linear.
}
    \label{fig:cplx2realD6}
\end{figure}

\section{Continuity at the knots}\label{sec:knots}
In this section we examine the behavior of the shadow density at the knots,
where the curve pieces meet, that is, the regular singular points of the
shadow differential equation. The even and odd $N$ cases are quite different.
For even $N$ there are even and odd segments based on counting from $a_{N}$,
so $\left[  a_{N-1},a_{N}\right]  $ is \#1, and this parity is the same if one
counts up from $a_{1}$. For odd $N$ there are even and odd knots (the parity
of $j$ for the knot $a_{j}$; this remains the same under the transformation
$x\mapsto-x$). In the neighborhood of each knot $a_{j}$ there is the analytic
part, expandable in a power series $\sum_{n=0}^{\infty}c_{n}\left(
x-a_{j}\right)  ^{n}$,
 and a part with discontinuous derivative of order
$\left\lfloor \frac{N}{2}\right\rfloor -1$, as will be shown.
For even $N=2M$
the density is polynomial of degree $M-2$ in $x$ in the even intervals
$\left[  a_{N-2j},a_{N-2j+1}\right]  $, and has a jump of the form $\left\vert
x-a_{i}\right\vert ^{M-\frac{3}{2}}$ at each end on the odd intervals $\left[
a_{N-2j-1},a_{N-2j}\right]  $. Recall that the critical exponents are
$\frac{N-1}{2}-1=M-\frac{3}{2}$ and $0,1,\ldots,N-3$.

For odd $N=2M+1$ there is just one type of curve piece: behavior like
$\left\vert x-a_{i}\right\vert ^{M-1}$ at the odd end-point and $\left\vert
x-a_{i}\right\vert ^{M-1}\log\left\vert x-a_{i}\right\vert $ at the even
end-point. The critical exponent $M-1=\frac{N-1}{2}-1$ is repeated, accounting
for the logarithmic term. In this case each interval can be considered as even
or odd by starting from $a_{N}$ or from $a_{1}$ and using 
the transformation $x\mapsto -x$.

For $a_{N-2m-1}\leq x<a_{N-2m}$ we have
\begin{align}
f\left(  x\right)   &  =\frac{N-2}{2\pi}\sum_{j=0}^{m-1}\left(  -1\right)
^{j}\int_{a_{N-2j-1}}^{a_{N-2j}}g_{j}\left(  s\right)  \left(  s-x\right)
^{\frac{N}{2}-2}ds\\
&  +\frac{N-2}{2\pi}\left(  -1\right)  ^{m}\int_{x}^{a_{N-2m}}g_{m}\left(
s\right)  \left(  s-x\right)  ^{\frac{N}{2}-2}ds. \nonumber
\end{align}
The integral indexed by $j$ is analytic in $x<a_{N-2j-1}$. Furthermore if
$N=2M$ then the sum defines a polynomial in $x$ without any further
restrictions on $x$. Consider the integral in the second line for
$a_{N-2m}-\varepsilon<x<a_{N-2m}$ for some small $\varepsilon>0$ (and
$<a_{N-2m}-a_{N-2m-1}$). Set $h_{m}\left(  s\right)  =g_{m}\left(  s\right)
\left(  a_{N-2m}-s\right)  ^{\frac{1}{2}}$ so that $h_{m}\left(  s\right)  $
has a power series expansion $\sum_{n=0}^{\infty}c_{n}\left(  a_{N-2m}%
-s\right)  ^{n}$ valid in a neighborhood of $\left[  a_{N-2m}-r,a_{N-2m}%
\right]  $ for small enough $r$. Then%
\begin{gather}
\int_{x}^{a_{N-2m}}g_{m}\left(  s\right)  \left(  s-x\right)  ^{\frac{N}{2}%
-2}ds=\int_{x}^{a_{N-2m}}\left(  a_{N-2m}-s\right)  ^{-\frac{1}{2}}\left(
s-x\right)  ^{\frac{N}{2}-2}h_{m}\left(  s\right)  ds\\ \nonumber
=\sum_{n=0}^{\infty}c_{n}\int_{x}^{a_{N-2m}}\left(  a_{N-2m}-s\right)
^{-\frac{1}{2}+n}\left(  s-x\right)  ^{\frac{N}{2}-2}ds\\ \nonumber
=\left(  a_{N-2m}-x\right)  ^{\frac{N-3}{2}}\sum_{n=0}^{\infty}B\left(
n+\frac{1}{2},\frac{N-2}{2}\right)  c_{n}\left(  a_{N-2m}-x\right)  ^{n}\\
=B\left(  \frac{1}{2},\frac{N-2}{2}\right)  \left(  a_{N-2m}-x\right)
^{\frac{N-3}{2}}\sum_{n=0}^{\infty}\frac{\left(  \frac{1}{2}\right)  _{n}%
}{\left(  \frac{N-1}{2}\right)  _{n}}c_{n}\left(  a_{N-2m}-x\right)  ^{n}; \nonumber
\end{gather}
this is the solution of the shadow equation for the critical exponent
$\frac{N-3}{2}$ at the regular singular point $a_{N-2m}$. The leading term is%
\begin{equation}
\frac{\left(  -1\right)  ^{m}}{B\left(  \frac{N-1}{2},\frac{1}{2}\right)
}\prod\limits_{j=0}^{2m-1}\left(  a_{N-j}-a_{N-2m}\right)  ^{-\frac{1}{2}%
}\prod\limits_{j=2m+1}^{N-1}\left(  a_{N-2m}-a_{N-j}\right)  ^{-\frac{1}{2}%
}\left(  a_{N-2m}-x\right)  ^{\frac{N-3}{2}},
\end{equation}
in analogy  to the leading term in formula (\ref{topF}). 
The computation uses an identity,
$\frac{1}{\pi}\left(  \frac{N}{2}-1\right)  \frac{\Gamma\left(  \frac{1}%
{2}\right)  \Gamma\left(  \frac{N-2}{2}\right)  }{\Gamma\left(  \frac{N-1}%
{2}\right)  }=\frac{\Gamma\left(  \frac{N}{2}\right)  }{\Gamma\left(
\frac{N-1}{2}\right)  \Gamma\left(  \frac{1}{2}\right)  }$.

\subsection{Even $N$}

Set $N=2M$. Near a knot $a_{2M-2m}$ ($0\leq m<M$) the density $f\left(
x\right)  $ is polynomial for $a_{2M-2m}<x$ and given by the sum of the
polynomial and a series $\left(  a_{2M-2m}-x\right)  ^{M-\frac{3}{2}}%
\sum_{n=0}^{\infty}c_{n}\left(  a_{2M-2m}-x\right)  ^{n}$ for $x<a_{2M-2m}$.
Thus $f^{\left(  j\right)  }\left(  x\right)  $ is continuous in a
neighborhood of $a_{2M-2m}$ for $0\leq j\leq M-2$. By applying this result to
the reversed knots $b_{1}<\ldots<b_{N}$ where $b_{j}=-a_{2M+1-j}$ and $x$
replaced by $-x$ we find that near a knot $a_{2M-2m+1}=-b_{2M-2\left(
M-m\right)  }$ for $1\leq m\leq M$ the density $f\left(  x\right)$ 
is polynomial for
$x<a_{2M-2m+1}$ and is given by the sum of the polynomial and a series
$\left(  x-a_{2M-2m+1}\right)  ^{M-\frac{3}{2}}\sum_{n=0}^{\infty}c_{n}\left(
x-a_{2M-2m+1}\right)  ^{n}$ for $x>a_{2M-2m+1}$. Thus the lowest order
discontinuity of the density is in $f^{\left(  M-1\right)  }\left(  x\right)
$ at each knot, that is, $f^{\left(  j\right)  }\left(  x\right)  $ is
continuous everywhere for all $j\leq\frac{N}{2}-2$.

\subsection{Odd $N$}

Set $N=2M+1$. Consider the even knot $a_{N-2m-1}$ with $0\leq m<M-1$. Pick $r$
with $0<r<\min\left(  a_{N-2m}-a_{N-2m-1},a_{N-2m-1}-a_{N-2m-2}\right)  $,
then in the interval $\left[  a_{N-2m-1}-r,a_{N-2m-1}+r\right]  $ the function 
\linebreak
$\prod\limits_{j=0}^{2m}\left(  a_{N-j}-s\right)  ^{-\frac{1}{2}}%
\prod\limits_{j=2m-2}^{N-1}\left(  s-a_{N-j}\right)  ^{-\frac{1}{2}}$ can be
expanded as a power series \linebreak
$\sum_{n=0}^{\infty}c_{n}\left(  s-a_{N-2m-1}%
\right)  ^{n}$. Here the coefficients can be found by using the negative binomial
theorem for each factor in the product.
 Since all the difficulty happens at
the knot set%
\begin{align}
\phi\left(  x\right)   &  =\frac{N-2}{2\pi}\sum_{j=0}^{m-1}\left(  -1\right)
^{j}\int_{a_{N-2j-1}}^{a_{N-2j}}g_{j}\left(  s\right)  \left(  s-x\right)
^{\frac{N}{2}-2}ds\\
&  +\frac{N-2}{2\pi}\left(  -1\right)  ^{m}\int_{a_{N-2m-1}+r}^{a_{N-2m}}%
g_{m}\left(  s\right)  \left(  s-x\right)  ^{\frac{N}{2}-2}ds. \nonumber
\end{align}
Thus $\phi\left(  x\right)  $ is analytic for $x<a_{N-2m-1}+r$; for
$a_{N-2m-1}<x<a_{N-2m-1}+r$%
\begin{equation}
f\left(  x\right)  =\phi\left(  x\right)  +\frac{N-2}{2\pi}\left(  -1\right)
^{m}\int_{x}^{a_{N-2m-1}+r}g_{m}\left(  s\right)  \left(  s-x\right)
^{M-\frac{3}{2}}ds,
\end{equation}
and for $a_{N-2m-1}-r<x<a_{N-2m-1}$%
\begin{equation}
f\left(  x\right)  =\phi\left(  x\right)  +\frac{N-2}{2\pi}\left(  -1\right)
^{m}\int_{a_{N-2m-1}}^{a_{N-2m-1}+r}g_{m}\left(  s\right)  \left(  s-x\right)
^{M-\frac{3}{2}}ds.
\end{equation}
By using the power series and the change of variable $x=a_{N-2m-1}+y$ and
$s=t+a_{N-2m-1}$the first integral becomes%
\begin{equation}
\sum_{n=0}^{\infty}c_{n}\int_{y}^{r}t^{n-\frac{1}{2}}\left(  t-y\right)
^{M-\frac{3}{2}}dt,0<y<r,
\end{equation}
and the second integral becomes%
\begin{equation}
\sum_{n=0}^{\infty}c_{n}\int_{0}^{r}t^{n-\frac{1}{2}}\left(  t-y\right)
^{M-\frac{3}{2}}dt,-r<y<0.
\end{equation}
We want to analyze the behavior of the integrals in the limit $y\rightarrow0$.
In each integral change the variable $t=\frac{y}{1-u^{2}}$, so $dt=\frac
{2yu}{\left(  1-u^{2}\right)  ^{2}}du$. Furthermore $t\left(  t-y\right)
=\frac{y^{2}u^{2}}{\left(  1-u^{2}\right)  ^{2}}$ thus $\frac{1}%
{\sqrt{t\left(  t-y\right)  }}=\frac{1-u^{2}}{yu}$ (if $y<0$ then $u^{2}>1$,
and if $y>0$ then $u^{2}<1$ so that this is the positive root). Set
$u_{r}=\sqrt{\frac{r-y}{r}}$. For $0<y<r$ the integral is
\begin{equation}
2y^{n+M-1}\int_{0}^{u_{r}}\frac{u^{2M-2}}{\left(  1-u^{2}\right)  ^{n+M}%
}du,~0<u_{r}<1,
\end{equation}
and for $-r<y<0$ the integral is
\begin{equation}
2y^{n+M-1}\int_{\infty}^{u_{r}}\frac{u^{2M-2}}{\left(  1-u^{2}\right)  ^{n+M}%
}du,~u_{r}>1.
\end{equation}
Because the integrand is even we deduce that the partial fraction expansion is
of the form%
\begin{equation}
\frac{u^{2M-2}}{\left(  1-u^{2}\right)  ^{n+M}}=\sum_{j=1}^{n+M}\beta
_{j}\left(  M,n\right)  \left\{  \frac{1}{\left(  1-u\right)  ^{j}}+\frac
{1}{\left(  1+u\right)  ^{j}}\right\}  ,
\end{equation}
for certain constants $\beta_{j}\left(  M,n\right)  $. Thus%
\begin{equation}
\begin{split}
I_{M,n}\left(  u\right)  := & \int\frac{u^{2M-2}}{\left(  1-u^{2}\right)  
^{n+M}%
}du=\beta_{1}\left(  M,n\right)  \log\left\vert \frac{1+u}{1-u}\right\vert + \\
& +\sum_{j=2}^{n+M}\frac{\beta_{j}\left(  M,n\right)  }{j-1}\left\{  \frac
{1}{\left(  1-u\right)  ^{j-1}}-\frac{1}{\left(  1+u\right)  ^{j-1}}\right\}
\end{split}.
\end{equation}
This antiderivative $I_{M,n}$ vanishes at $u=0$ and at $u=\infty$, thus both
integrals have the same value $2y^{n+M-1}I_{M,n}\left(  \sqrt{\frac{r-y}{r}%
}\right)  $. The terms for $2\leq j\leq n+M$ contribute%
\begin{equation}
\frac{4y^{n+M-1}u_{r}}{\left(  1-u_{r}^{2}\right)  ^{j-1}}\sum_{i=0}%
^{\left\lfloor \left(  j-2\right)  /2\right\rfloor }\binom{j-1}{2i+1}%
u_{r}^{2i}=4y^{n+M-j}\sqrt{\frac{r-y}{r}}\sum_{i=0}^{\left\lfloor \left(
j-2\right)  /2\right\rfloor }\binom{j-1}{2i+1}\left(  r-y\right)
^{i}r^{j-1-i},
\end{equation}
which is analytic in $y$ for $-r<y<r$. So all the singular behavior stems from
the logarithmic term%
\begin{equation}
2y^{n+M-1}\beta_{1}\left(  M,n\right)  \log\left\vert \frac{1+u_{r}}{1-u_{r}%
}\right\vert ,
\end{equation}
and%
\begin{align}
\log\left\vert \frac{1+u_{r}}{1-u_{r}}\right\vert  &  =-\log\left\vert
1-u_{r}^{2}\right\vert +2\log\left\vert 1+u_{r}\right\vert \\
&  =-\log\left\vert y\right\vert +\log r+2\log(1+\sqrt{\frac{r-y}{r}}). \nonumber
\end{align}
Collecting the relevant terms we see that for $a_{N-2m-1}-r<x<a_{N-2m-1}+r$
the density $f\left(  x\right)  $ is the sum of an analytic part and%
\begin{equation}
\begin{split}
& \frac{2}{\pi}\left(  -1\right)  ^{m-1}\left(  2M-1\right)  \log\left\vert
x-a_{N-2m-1}\right\vert \times \\
& \times \left(  x-a_{N-2m-1}\right)  ^{M-1}\sum_{n=0}^{\infty
}\beta_{1}\left(  M,n\right)  c_{n}\left(  x-a_{N-2m-1}\right)  ^{n}.
\end{split}
\end{equation}

The coefficients $\beta_{j}\left(  M,n\right)  $ can be found explicitly as
sums but we are only concerned with $\beta_{1}\left(  M,n\right)  $. Indeed
(proof left for reader)%
\begin{equation}
\beta_{1}\left(  M,n\right)  =\frac{1}{2}\left(  -1\right)  ^{M-1}%
\frac{\left(  \frac{1}{2}\right)  _{M-1}\left(  \frac{1}{2}\right)  _{n}%
}{\left(  n+M-1\right)  !}.
\end{equation}
Thus we analyzed the behavior at the even knots and showed that $f^{\left(
j\right)  }\left(  x\right)  $ is continuous everywhere for all $j\leq
M-2=\frac{N-1}{2}-2$.

\section{Entangled shadow}\label{sec:entangled-shadow}
In previous sections we investigated the shadow with respect to
real states. Here we discuss another example of the restricted shadow --
the shadow with respect to maximally entangled states,
 briefly called {\sl entangled shadow}.

\medskip 
\subsection{Entangled shadow of $4 \times 4$ matrices with direct sum 
structure}

We shall start recalling the definition of the entangled shadow 
introduced in \cite{shadow3}.
\begin{definition}
Maximally entangled numerical shadow of a matrix $A$ of size $N=N_1\times N_2$ 
is defined as a probability distribution $\mathcal{P}^{\mathrm{ent}}_A(z)$ on the complex plane.
\begin{equation}
\mathcal{P}_A^{\mathrm{ent}}(z) := \int {\mathrm d} \mu(\psi) \delta\Bigl( z-\langle 
\psi|A|\psi\rangle\Bigr),
\end{equation}
where $\mu(\psi)$ denotes the unique unitarily invariant (Fubini-Study) 
measure on the set of complex pure states restricted to the set of bi-partite maximally entangled states
\begin{equation}
\left\{\ket{\psi} \in \mathbb{C}^{N_1\times N_2}:
\ket{\psi}=\frac{1}{\sqrt{N_{\min}}} (U_1\otimes U_2)\sum_{i=1}^{N_{\min}} 
\ket{\psi_i^1}\otimes \ket{\psi_i^2} \right\}.
\end{equation}
Here  $N_{\min}=\min(N_1,N_2)$, 
while
$\ket{\psi_i^1}$, $\ket{\psi_i^2}$ form orthonormal bases in 
$\mathbb{C}^{N_1}$ and $\mathbb{C}^{N_2}$ respectively,
while $U_1\in SU(N_1)$ and $U_2\in SU(N_2)$.
\end{definition}

\begin{definition}
Pauli matrices $\sigma_x$, $\sigma_y$ and $\sigma_z$ are defined as
\begin{equation}
\sigma_x = 
\begin{pmatrix}
0 & 1 \\
1 & 0
\end{pmatrix},
\;
\sigma_y = 
\begin{pmatrix}
0 & -\ii \\
\ii & 0
\end{pmatrix},
\;
\sigma_z = 
\begin{pmatrix}
1 & 0 \\
0 & -1
\end{pmatrix}.
\end{equation}
\end{definition}

\begin{lemma}
For a $2 \times 2$ unitary matrix $U$ and an {\sl arbitrary} $2 
\times 2$ matrix $A$ we have
\begin{equation}
\bra{1} U^{\dagger} A U \ket{1} = \bra{0} U^{\dagger} \sigma_y A^\mathrm{T} 
\sigma_y U \ket{0},
\end{equation}
\end{lemma}

\begin{theorem}
Maximally entangled shadow
\begin{equation}
\mathcal{P}^{\mathrm{ent}}_{A \oplus B} = \mathcal{P}_{\frac12 (A + \sigma_y B^\mathrm{T} \sigma_y)},
\end{equation}
where $A\oplus B$ denotes block matrix.
\end{theorem}

\begin{proof}
We write
\begin{equation}
\begin{split}
\bra{\psi} (A \oplus B) \ket{\psi} &= 
\bra{\psi_+} (\1 \otimes U^{\dagger}) (A \oplus B) (\1 \otimes U)  
\ket{\psi_+} \\
&=\bra{\psi_+} (U^{\dagger} \oplus U^{\dagger}) (A \oplus B) (U \oplus U)  
\ket{\psi_+} \\
%&=\bra{\psi_+} ((U^{\dagger}AU)\oplus (U^{\dagger} B U)  \ket{\psi_+}\\
&= \frac12 \left(\bra{0}U^{\dagger}AU \ket{0} + \bra{1}U^{\dagger}BU 
\ket{1}\right).
\end{split}
\end{equation}
Now we use lemma and write
\begin{equation}
\begin{split}
\bra{\psi} (A \oplus B) \ket{\psi} &= \frac12 \left(\bra{0}U^{\dagger}AU 
\ket{0} + \bra{1}U^{\dagger}BU \ket{1}\right)\\
&= \frac12 \left(\bra{0}U^{\dagger}AU \ket{0} + \bra{0}U^{\dagger}\sigma_y 
B^\mathrm{T} 
\sigma_yU 
\ket{0}\right)\\
&= \frac12 \left(\bra{0}U^{\dagger}(A+\sigma_y B^\mathrm{T} 
\sigma_y)U \ket{0}\right).
\end{split}
\end{equation}
\end{proof}

This theorem is valid for complex and real entangled shadow.
 Also we made no assumptions 
on $A$ and $B$, hence it is valid for non-normal matrices.

\subsection{Real maximally entangled shadow of $4 \times 4$ matrices}

\begin{definition}
Real maximally entangled numerical shadow 
$\mathcal{P}_A^{{\mathrm{ent} | \mathbb{R}}}$ 
of a matrix $A$ of size $N=N_1\times 
N_2$ is defined similarly to the maximally entangled shadow, but with restriction
to the real maximally entangled states.
\end{definition}

The following theorem gives a full characterization of the real maximally 
entangled numerical shadow of $4 \times 4$ matrices.
\begin{theorem}\label{th:real-ent}
Let $A$ be any $4 \times 4$ matrix 
then we have
\begin{equation}
\mathcal{P}_A^{\mathrm{ent} | \mathbb{R}}=
\frac{1}{2} \mathcal{P}_{Z_1^\mathrm{T} A Z_1}^{\mu_\mathbb{R}}+
\frac{1}{2} \mathcal{P}_{Z_2^\mathrm{T} A Z_2}^{\mu_\mathbb{R}},
\end{equation}
where, $Z_1$ and $Z_2$ are
\begin{equation}
Z_1 = \frac{1}{\sqrt 2}
\left(
\begin{smallmatrix}
1 & 0 \\
0 & 1 \\
0 & -1  \\
1 & 0
\end{smallmatrix}
\right) ,
 \ \
Z_2 = \frac{1}{\sqrt 2}
\left(
\begin{smallmatrix}
1 & 0 \\
0 & 1 \\
0 & 1  \\
-1 & 0
\end{smallmatrix}
\right).
\label{z1z2}
\end{equation}
\end{theorem}

\begin{proof}
Any real maximally entangled pure state $\ket{\psi}$ of size four
may be written as a vector obtained from the elements
of an orthogonal matrix of order two, 
\begin{equation}
\ket{\psi} =\frac{1}{\sqrt 2} \mathrm{vec}(O(\theta)).
\end{equation}
First we consider an  orthogonal matrix $O(\theta)$ satisfying $\det O(\theta) 
= 1$. We have
\begin{equation}
\ket{\psi} = 
\frac{1}{\sqrt 2}
\left(
\begin{smallmatrix}
\cos \theta \\
\sin \theta \\
-\sin \theta \\
\cos \theta
\end{smallmatrix}
\right)
=
\frac{1}{\sqrt 2}
\left(
\begin{smallmatrix}
1 & 0 \\
0 & 1 \\
0 & -1 \\
1 & 0
\end{smallmatrix}
\right)
\left(
\begin{smallmatrix}
\cos \theta \\
\sin \theta
\end{smallmatrix}
\right).
\end{equation}
Hence,
\begin{equation}
\bra{\psi}A\ket{\psi} = \bra{r}Z_1^\mathrm{T} A Z_1 \ket{r}. 
\label{eq:max-ent-p1}
\end{equation}

Now we consider an  orthogonal matrix $O(\theta)$ satisfying $\det O(\theta) 
= -1$. We have
\begin{equation}
\ket{\psi} = 
\frac{1}{\sqrt 2}
\left(
\begin{smallmatrix}
\cos \theta \\
\sin \theta \\
\sin \theta \\
-\cos \theta
\end{smallmatrix}
\right)
=
\frac{1}{\sqrt 2}
\left(
\begin{smallmatrix}
1 & 0 \\
0 & 1 \\
0 & 1 \\
-1 & 0
\end{smallmatrix}
\right)
\left(
\begin{smallmatrix}
\cos \theta \\
\sin \theta
\end{smallmatrix}
\right).
\end{equation}
Hence,
\begin{equation}
\bra{\psi}A\ket{\psi} = \bra{r}Z_2^\mathrm{T} A Z_2 \ket{r}. 
\label{eq:max-ent-p2}
\end{equation}
Combining equations \eqref{eq:max-ent-p1} and \eqref{eq:max-ent-p2} we get the 
theorem.
\end{proof}
\medskip 

Examples are shown in Figures~\ref{fig:t2-ex-1} and \ref{fig:t2-ex-2}. The 
theorem is valid for non-normal matrices.
\begin{figure}[!h]
\subfloat[Real maximally entangled 
shadow]{\includegraphics{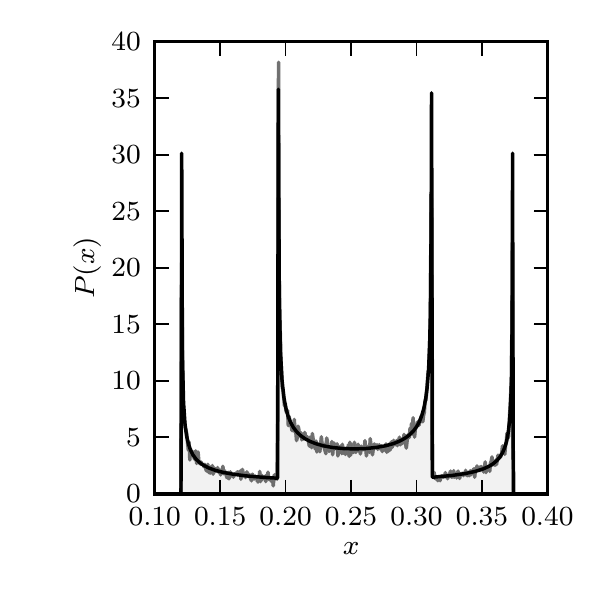}\label{fig:t2-ex-1}}
\subfloat[Complex maximally entangled 
shadow]{\includegraphics{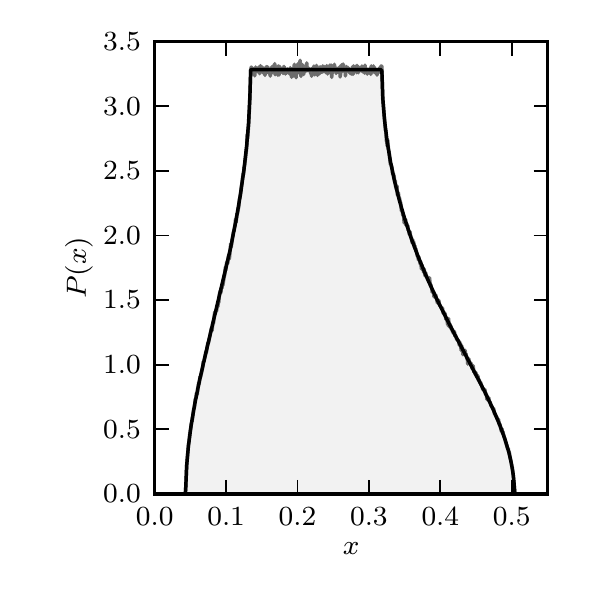}\label{fig:t2-ex-2}}
\caption{Fig a: Real maximally entangled numerical shadow of a
Hermitian matrix sampled from the Hilbert-Schmidt distribution.
Monte Carlo integration result (light gray) and shadow composition based on 
Theorem~\ref{th:real-ent} (black). Fig. b: complex maximally entangled 
numerical shadow of a generic Hermitian matrix.
Monte Carlo integration result (light gray) and shadow composition based on 
Theorem~\ref{th:complex-ent} (black). The difference in heights of the peaks 
(case a) and fluctuations at the central segment of the spectrum (case b)
 are due to numerical errors.}
\end{figure}
%[trim=left bottom right top, clip]
\begin{figure}[h!]
    \centering
    \begin{tabular}{m{0.5\linewidth}m{0.5\linewidth}}
    \subfloat[Real numerical shadow of matrix $Z_1AZ_1^\dagger$ 
obtained using Monte-Carlo sampling.]
{\includegraphics[trim=4mm 4mm 7mm 5mm, clip]{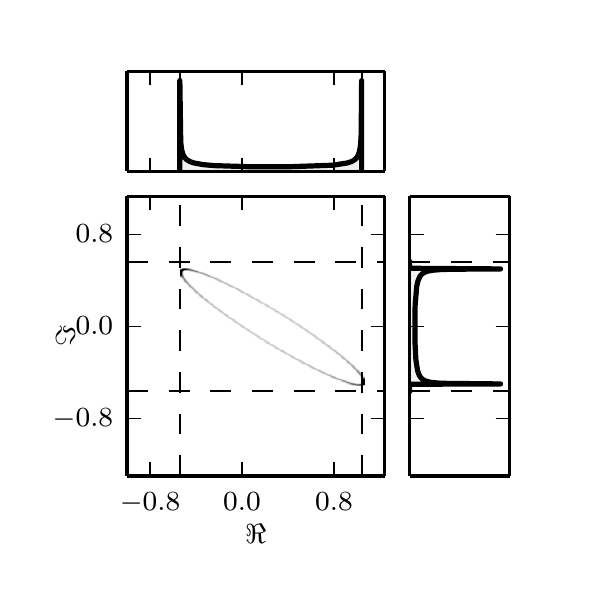}\label{}} &
    \subfloat[Real numerical shadow of matrix $Z_2AZ_2^\dagger$ 
obtained using Monte-Carlo sampling.]
{\includegraphics[trim=4mm 4mm 7mm 5mm, clip]{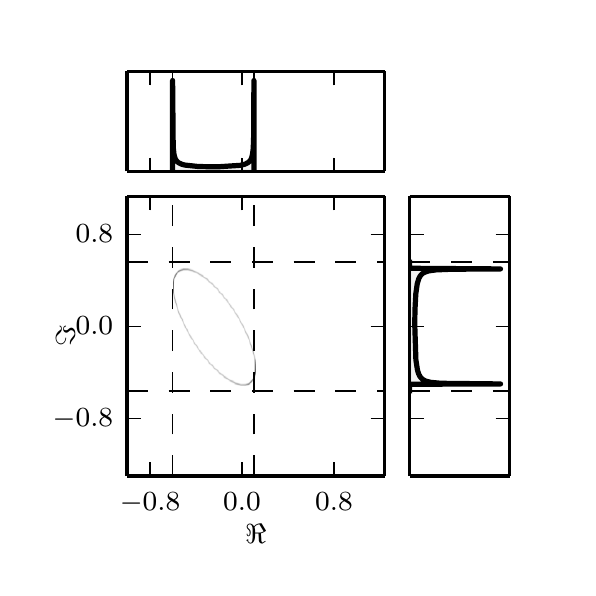}\label{}} \\
    \subfloat[Real entangled numerical shadow of matrix $A$ using Monte-Carlo sampling.]             {\includegraphics[trim=4mm 4mm 7mm 5mm, clip]{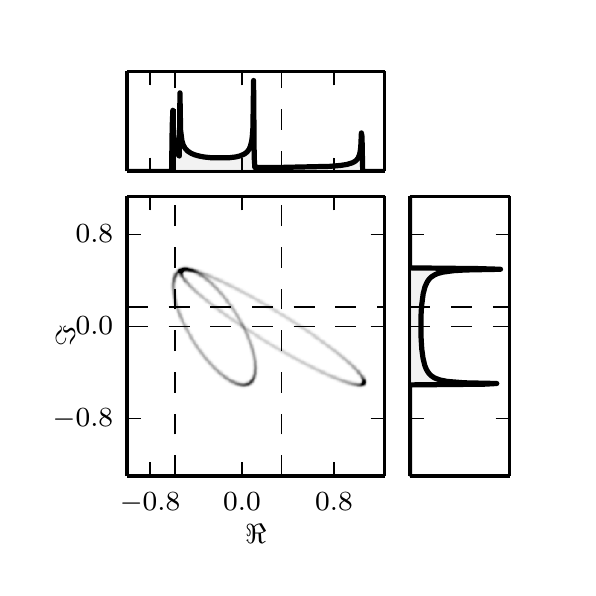}\label{}} &
        $
        A=
        \left(
        \begin{smallmatrix}
        1 & 0 & 0 & 1 \\
        0 & \ii & 0 & 1 \\
        0 & 0 & -1 & 0\\
        0 & 0 & 0 & -\ii
        \end{smallmatrix}
        \right)
        $
    
    \end{tabular}
    \caption[Visualisation of Theorem~\ref{th:real-ent} --- $A$.]{Visualisation of Theorem~\ref{th:real-ent} using matrix $A$.}
\end{figure}
%[trim=left bottom right top, clip]
\begin{figure}[h!]
    \centering
    \begin{tabular}{m{0.5\linewidth}m{0.5\linewidth}}
    \subfloat[Real numerical shadow of matrix $Z_1BZ_1^\dagger$ obtained using 
    Monte-Carlo sampling.]{\includegraphics[trim=4mm 8mm 7mm 11mm, 
    clip]{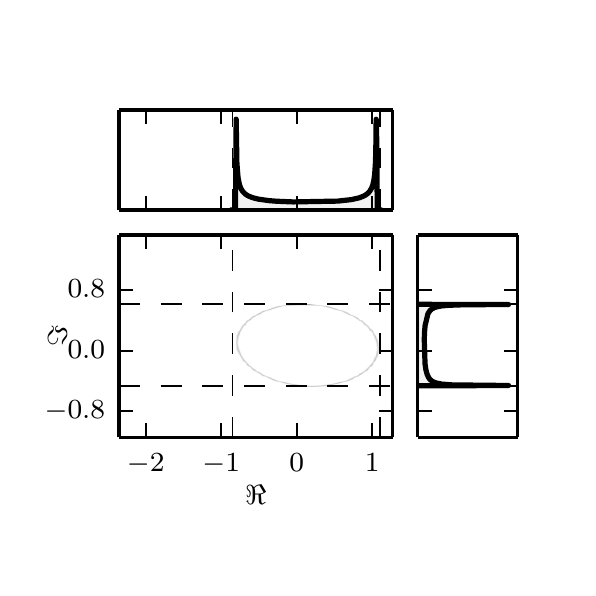}\label{}} &
    \subfloat[Real numerical shadow of matrix $Z_2BZ_2^\dagger$ obtained using 
    Monte-Carlo sampling.]{\includegraphics[trim=4mm 8mm 7mm 11mm, 
    clip]{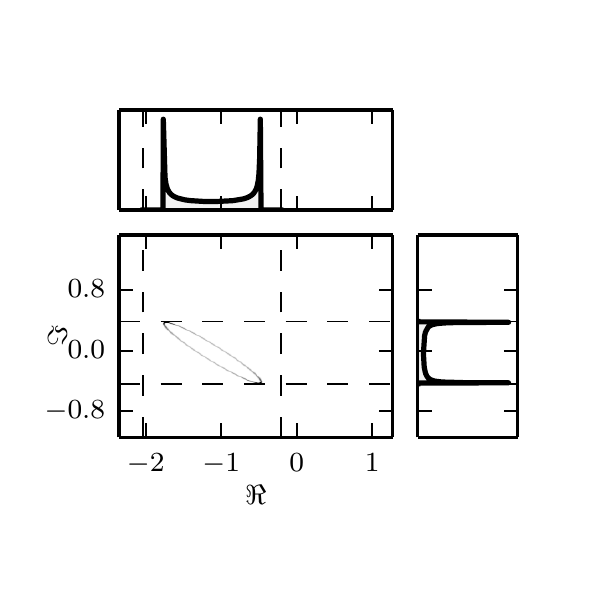}\label{}} \\
    \subfloat[Real entangled numerical shadow of matrix $B$ using Monte-Carlo sampling.]             {\includegraphics[trim=4mm 8mm 7mm 11mm, clip]{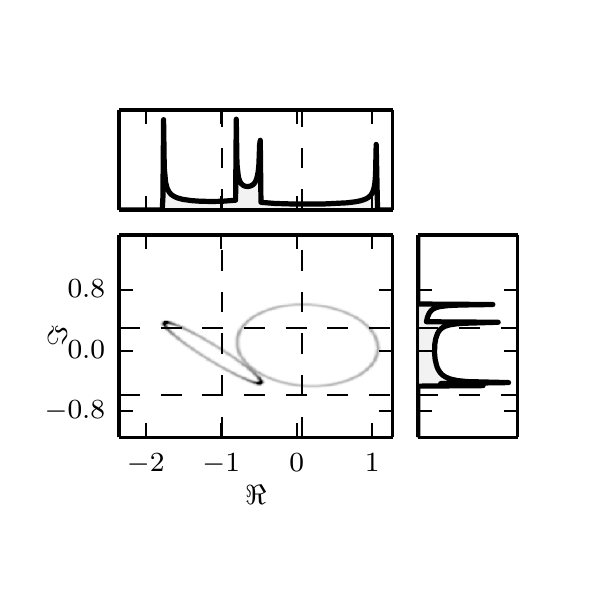}\label{}} &
    $B=\left(
            \begin{smallmatrix}
             0.3+0.5\ii & -0.8-0.2\ii & 0.4-0.5\ii & 1\\
            0.6-0.8\ii & -0.8-0.4\ii & -0.6+0.8\ii & -0.8+0.8\ii\\
            0.7-0.8\ii & -0.5-0.4\ii & -0.8 & 0.7-0.3\ii\\
            0.4+0.6\ii & -1.-0.8\ii & -0.4-0.4\ii & -0.7
            \end{smallmatrix}
            \right)$
    \end{tabular}
    \caption[Visualisation of Theorem~\ref{th:real-ent} --- $B$.]{
    Visualisation of Theorem~\ref{th:real-ent} using matrix $B$.
    }
\end{figure}
\subsection{Complex maximally entangled shadow of $4 \times 4$ matrices}
\begin{theorem}\label{th:complex-ent}
Given an {\sl arbitrary} matrix $A$ of order four 
its complex maximally entangled shadow is equal to the 
real shadow of matrix $W^\dagger A W$
\begin{equation}
\mathcal{P}^{\mu_\mathrm{ent} }_A = \mathcal{P}^{\mu_\mathbb{R 
}}_{W^\dagger A W},
\end{equation}
where $W$ is the matrix representing the 'magic basis',
\begin{equation}
W = \frac{1}{\sqrt{2}}
\begin{pmatrix}
0 & 0 & 1 & \ii \\
-1 & \ii & 0 & 0 \\
1 & \ii & 0 & 0 \\
0 & 0 & 1 & -\ii
\end{pmatrix}.
\end{equation}
\end{theorem}
The above theorem is related to the well known fact in the group theory, that 
\begin{equation}
SO(4) = (SU(2) \times SU(2)) / \mathbb{Z}_2.
\end{equation}
\begin{proof}
Any maximally entangled two-qubit state $\ket{\psi}$ can be written as
\begin{equation}
\ket{\psi} = \mathrm{vec}(V), \ {\rm where} \  V \in SU(2).
\end{equation}
Using a parameterization of $SU(2)$ we can write
\begin{equation}
V =
\begin{pmatrix}
\ee^{\ii \xi_2} \cos\eta & \ee^{\ii \xi_1} \sin \eta \\
-\ee^{-\ii \xi_1} \sin \eta & \ee^{-\ii \xi_2} \cos\eta
\end{pmatrix}.
\end{equation}
Reshaping this matrix into a vector of length four
we obtain the state 
\begin{equation}
\ket{\psi} = \frac{1}{\sqrt{2}}
\begin{pmatrix}
\cos \xi_2 \sin \eta + \ii \sin \xi_2 \cos \eta \\
-\cos \xi_1 \sin \eta + \ii \sin \xi_1 \sin \eta \\
\cos \xi_1 \sin \eta + \ii \sin\xi_1 \sin \eta \\
\cos \xi_2 \cos \eta - \ii \sin \xi_2 \cos \eta
\end{pmatrix}.\label{eq:vectorized}
\end{equation}

On the other hand, consider the Hopf parameterization of the 3-sphere $S^3$ 
embedded in $\C^2$. A point on this sphere can be expressed as
\begin{equation}
\begin{cases}
z_1  =  \ee^{\ii \xi_1} \sin \eta \\
z_2  =  \ee^{\ii \xi_2} \cos \eta.
\end{cases}
\end{equation}

A point on the 3-sphere may be written in real coordinates $(r_1, r_2, r_3, 
r_4)$ as
\begin{equation}
\begin{pmatrix}
r_1 \\
r_2 \\
r_3 \\
r_4 \\
\end{pmatrix}
=
\begin{pmatrix}
\Re z_1 \\
\Im z_1 \\
\Re z_2 \\
\Im z_2
\end{pmatrix}
=
\begin{pmatrix}
\cos \xi_1 \sin \eta \\
\sin \xi_1 \sin \eta \\
\cos \xi_2 \cos \eta \\
\sin \xi_2 \cos \eta
\end{pmatrix}.
\label{eq:real-vec}
\end{equation}
Now, using Equation~\eqref{eq:real-vec}, we can rewrite 
Equation~\eqref{eq:vectorized} as
\begin{equation}
\ket{\psi} =
\frac{1}{\sqrt{2}}
\begin{pmatrix}
r_3 +\ii r_4 \\
-r_1 + \ii r_2 \\
r_1 + \ii r_2 \\
r_3 - \ii r_4
\end{pmatrix}
=
\frac{1}{\sqrt{2}}
\begin{pmatrix}
0 & 0 & 1 & \ii \\
-1 & \ii & 0 & 0 \\
1 & \ii & 0 & 0 \\
0 & 0 & 1 & -\ii
\end{pmatrix}
\begin{pmatrix}
r_1\\
r_2\\
r_3\\
r_4
\end{pmatrix}.
\end{equation}
Hence, we can write
\begin{equation}
\bra{\psi} A \ket{\psi} = \bra{r} W^\dagger A W \ket{r},
\end{equation}
where $\ket{r}$ is a real vector defined in equation~\eqref{eq:real-vec}.
\end{proof} 

\section{Concluding remarks}\label{sec:remarks}

In this work we analyzed probability distributions on the
complex plane induced by projecting the set of quantum states,
(i.e. Hermitian, positive and  normalized matrices of a given size $N$),
endowed with a certain probability measure. In the 
case of the unique, unitarily invariant Haar measure
on the set of complex pure states, this distribution
coincides with the standard numerical shadow \cite{Shd1,GS2012} of
a certain matrix $A$ of size $N$.
The case of a normal matrix corresponds to the projection
of the unit simplex covered uniformly onto a plane \cite{shadow1}.
If the matrix $A$ is Hermitian, its (complex) numerical shadow is supported on 
an interval on the real axis, and is equivalent to the $B$--spline 
with knots at the eigenvalues of $A$.

The real shadow of a matrix corresponds to the Haar 
measure restricted to the set of real pure states \cite{shadow3}.
For a real symmetric $A$ its real shadow is shown to 
be equivalent to a to the projection
of the unit simplex covered by the Dirichlet measure.
The main result of this work consists in 
Theorem~\ref{th:main-theorem}, which establishes 
an explicit exact formula for the real shadow of any real 
symmetric $A$ with prescribed spectrum $(a_1,\dots, a_N)$.

As the real shadow of a matrix corresponds to the Dirichlet distribution
with its parameters equal to $k_1= k_2=\dots=k_N=1/2$, it is natural to 
generalize
it by considering also other values of this parameter. For instance, the case 
of complex shadow $\mathcal{P}_A=\mathcal{P}^{\mathbb{C}}_A$
corresponds to the case $k_i=1$.
This fact implies that the real shadow of an extended 
matrix is equivalent to the complex shadow, 
\begin{equation}
\mathcal{P}^{\mathbb{R}}_{A\otimes {\mathbb I_2}}(x) =
\mathcal{P}_A(x) .
\label{prodshad}
\end{equation}
This result allows us to consider the real shadow of a real symmetric  matrix
$C$ of an even size $2N$ as a generalization of the $B$ spline, which is recovered,
if each eigenvalue is doubly degenerated. In general, each of $N$ knot points
of the standard $B$--spline can be splitted into two halves, 
and each eigenvalue $\lambda_i$ of $C$
 can be considered as a 'half of the knot point', 
as $2N$ points $\{ \lambda_i\}_{i=1}^{2N}$
determine the generalized $B$-spline equal to the real shadow of $C$.

Analyzing the generalized Dirichlet distribution one needs not to
restrict the attention to parameters $k_i$ equal to $1/2$ or $1$.
For instance, one can  consider the shadow of a matrix 
of an even order with respect to quaternion states
which  corresponds to the Dirichlet distribution with
all parameters equal,  $k_i=2$, see \ref{sec:quaternion-shadow}.

As another example of the restricted shadow we analyzed
entangled shadow of a matrix of an order $N$ equal to a composite
number. As before we distinguish the shadow with respect to
complex (or real) maximally entangled states.
Note that these probability distributions 
in general are supported on non--convex sets.
In the simplest case of $N=4$ we found explicit
formulae for the complex and real  entangled shadows
 by relating it to the real shadows of suitably transformed matrices.
As such shadows visualize projection of the set of complex/real
maximally entangled states onto a plane \cite{shadow3}
it is likely to expect that such tools will be useful in studying the
structure of the set of maximally entangled states.

\bigskip 
Acknowledgements:
It is a pleasure to thank John Holbrook for several discussions
on real numerical shadow and for providing us a copy of his paper
prior to publication.
This research was supported by the
the Polish National Science Centre (NCN):
P. Gawron under the grant number N N516 481840, Z.Pucha{\l}a under the grant 
number DEC-2012/04/S/ST6/00400 while K.~{\.Z}yczkowski and {\L}. Pawela 
acknowledge support by the grant number N202 090239.

\newpage
\appendix
\section{Quaternion shadow} \label{sec:quaternion-shadow}

In the case of the real shadow one considers random normalized real
vectors with distribution invariant to orthogonal transformations.
This distribution is induced by a Haar measure on the orthogonal group.
In a similar fashion, one can introduce the quaternion shadow, 
$\mathcal{P}^{\mathbb{H}}_A$,
defined as a probability distribution of expectation values taken among
random normalized quaternion vectors, with distribution invariant with
respect to symplectic operations.

An $N$-vector with quaternion $\mathbb{H}$ entries is replaced by a
$2N\times2$ complex matrix, where 
$a_{0}+a_{1}\mathbf{i}+a_{2}\mathbf{j}+a_{3}\mathbf{k}$ is mapped to%
\begin{equation}
\nu\left(a_{0}+a_{1}\mathbf{i}+a_{2}\mathbf{j}+a_{3}\mathbf{k}\right)
=\left[
\begin{array}
[c]{cc}%
a_{0}+a_{1}\mathrm{i} & a_{2}+a_{3}\mathrm{i}\\
-a_{2}+a_{3}\mathrm{i} & a_{0}-a_{1}\mathrm{i}%
\end{array}
\right]  .
\end{equation}
We use $\nu$ to also indicate the map $N$-vectors to $2N\times2$ matrices.
Suppose $A$ is a $2N\times2N$ complex Hermitian matrix. Consider the numerical
range-type map from the unit sphere in $\mathbb{H}^{N}$ to $\mathbb{R}:$%
\begin{equation}
\xi\mapsto\frac{1}{2}\Tr\left(  \nu\left(  \xi\right)  ^{\dagger}A\nu\left(
\xi\right)  \right)  .
\end{equation}
Note $\nu\left(  \xi\right)  ^{\dagger}A\nu\left(  \xi\right)  $ is a $2\times2$
complex Hermitian matrix. By direct computation we find that the same values
are obtained if $A$ is transformed as follows:%
\begin{equation}
q\left(\left[
\begin{array}
[c]{cc}%
a_{11} & a_{12}\\
a_{21} & a_{22}%
\end{array}
\right]\right)  =
\frac{1}{2}
\left[
\begin{array}
[c]{cc}%
\left(  a_{11}+\overline{a_{22}}\right)  & \left(
a_{12}-\overline{a_{21}}\right) \\
\left(  a_{21}-\overline{a_{12}}\right)  & \left(
\overline{a_{11}}+a_{22}\right)
\end{array}
\right]  ,
\end{equation}
and this map is applied to each of the $N^{2}$ $2\times2$ blocks of $A$. Note
that these blocks correspond to quaternions of the form%
\begin{equation}
\left[
\begin{array}
[c]{cc}%
\alpha & \beta\\
-\overline{\beta} & \overline{\alpha}%
\end{array}
\right]  .
\end{equation}
Then the numerical range and shadow can be interpreted as those of an $N\times
N$ quaternionic Hermitian matrix. The probability density is Dirichlet,
parameter 2, using $N$ real eigenvalues. The transformed matrix $q\left(
A\right)  $ has (duplicates) pairs of eigenvalues.

Here is a trivial example:%
\begin{equation}
A=\left[
\begin{array}
[c]{cccc}%
0 & 1 & 0 & 0\\
1 & 0 & 1 & 0\\
0 & 1 & 0 & 1\\
0 & 0 & 1 & 0
\end{array}
\right]  ,q\left(  A\right)  =\left[
\begin{array}
[c]{cccc}%
0 & 0 & 0 & -\frac{1}{2}\\
0 & 0 & \frac{1}{2} & 0\\
0 & \frac{1}{2} & 0 & 0\\
-\frac{1}{2} & 0 & 0 & 0
\end{array}
\right]  .
\end{equation}
The eigenvalues of $A$ are $\pm\frac{1}{2}\pm\frac{1}{2}\sqrt{5}$ and the
eigenvalues of $q\left(  A\right)  $ are $-\frac{1}{2},-\frac{1}{2},\frac
{1}{2},\frac{1}{2}$.

Consider $M_{2}\left(  \mathbb{C}\right)  $ as an eight-dimensional vector
space over $\mathbb{R}$ with the basis:%
\begin{equation}
\begin{split}
\mathbf{1} &  =\left[
\begin{array}
[c]{cc}%
1 & 0\\
0 & 1
\end{array}
\right]  ,\mathbf{i}=\left[
\begin{array}
[c]{cc}%
\mathrm{i} & 0\\
0 & -\mathrm{i}%
\end{array}
\right]  ,\mathbf{j}=\left[
\begin{array}
[c]{cc}%
0 & 1\\
-1 & 0
\end{array}
\right]  ,\mathbf{k}=\left[
\begin{array}
[c]{cc}%
0 & \mathrm{i}\\
\mathrm{i} & 0
\end{array}
\right]  ,\\
\mathbf{\zeta} &  =\left[
\begin{array}
[c]{cc}%
\mathrm{i} & 0\\
0 & \mathrm{i}%
\end{array}
\right]  ,\mathbf{\zeta i}=\left[
\begin{array}
[c]{cc}%
-1 & 0\\
0 & 1
\end{array}
\right]  ,\mathbf{\zeta j}=\left[
\begin{array}
[c]{cc}%
0 & \mathrm{i}\\
-\mathrm{i} & 0
\end{array}
\right]  ,\mathbf{\zeta k}=\left[
\begin{array}
[c]{cc}%
0 & -1\\
-1 & 0
\end{array}
\right]  ,
\end{split}
\end{equation}
(Pauli matrices). The basis is orthonormal with the inner product
\begin{equation}
\left\langle \alpha,\beta\right\rangle =\frac{1}{2}\operatorname{Re}\Tr\left(
\alpha\beta^{\dagger}\right)  ,
\end{equation}
and $\left\langle \alpha,\beta\right\rangle =\left\langle \alpha^{\dagger}%
,\beta^{\dagger}\right\rangle $. Then $q\left(  \alpha\right)  =\left\langle
\alpha,\mathbf{1}\right\rangle \mathbf{1+}\left\langle \alpha,\mathbf{i}%
\right\rangle \mathbf{i+}\left\langle \alpha,\mathbf{j}\right\rangle
\mathbf{j+}\left\langle \alpha,\mathbf{k}\right\rangle \mathbf{k}$ and
$q\left(  \alpha^{\dagger}\right)  =q\left(  \alpha\right)  ^{\dagger}$. Set
$\mathbb{H}=\mathrm{span}_{\mathbb{R}}\left\{  \mathbf{1},\mathbf{i}%
,\mathbf{j},\mathbf{k}\right\}  $ (a $\ast$-subalgebra). Now suppose $\left(
\alpha_{ij}\right)  _{i,j=1}^{N}$ is a Hermitian matrix with entries in
$M_{2}\left(  \mathbb{C}\right)  $ (that is, $\alpha_{ji}=\alpha_{ij}^{\dagger}%
$), and $\left(  \beta_{i}\right)  _{i=1}^{N}$ is a vector with entries in
$\mathbb{H}$. 
\begin{lemma}
The following equality holds
\begin{equation}
\Tr\left(  \sum_{i,j=1}^{N}\beta_{i}^{\dagger}\alpha_{ij}\beta_{j}\right)
=\Tr\left(  \sum_{i,j=1}^{N}\beta_{i}^{\dagger}q\left(  \alpha_{ij}\right)
\beta_{j}\right)  .
\end{equation}
\end{lemma}
\begin{proof}
Break up the sum into $i=j$ and $i<j$ parts. Then $\Tr\left(  \beta_{i}^{\dagger
}\alpha_{ii}\beta_{i}\right)  \in\mathbb{R}$ and%
\begin{equation}
\begin{split}
\Tr\left(  \beta_{i}^{\dagger}\alpha_{ii}\beta_{i}\right)   &  =
\Tr\left(
\alpha_{ii}\beta_{i}\beta_{i}^{\dagger}\right)  =2\left\langle \alpha_{ii}%
,\beta_{i}\beta_{i}^{\dagger}\right\rangle \\
&  =2\left\langle q\left(  \alpha_{ii}\right)  ,\beta_{i}\beta_{i}^{\dagger
}\right\rangle =\Tr\left(  \beta_{i}^{\dagger}q\left(  \alpha_{ii}\right)
\beta_{i}\right)  .
\end{split}
\end{equation}
For $i<j$ consider the typical term%
\begin{equation}
\begin{split}
\Tr\left(  \beta_{i}^{\dagger}\alpha_{ij}\beta_{j}\right)  +\Tr\left(  \beta
_{j}^{\dagger}\alpha_{ji}\beta_{i}\right)   &  =\Tr\left(  \beta_{i}^{\dagger}%
\alpha_{ij}\beta_{j}\right)  +\Tr\left(  \beta_{j}^{\dagger}\alpha_{ij}^{\dagger
}\beta_{i}\right)  \\
&  =2\operatorname{Re}\Tr\left(  \beta_{i}^{\dagger}\alpha_{ij}\beta_{j}\right)
=2\operatorname{Re}\Tr\left(  \alpha_{ij}\beta_{j}\beta_{i}^{\dagger}\right)  \\
&  =4\left\langle \alpha_{ij},\beta_{i}\beta_{j}^{\dagger}\right\rangle
=4\left\langle q\left(  \alpha_{ij}\right)  ,\beta_{i}\beta_{j}^{\dagger
}\right\rangle \\
&  =\Tr\left(  \beta_{i}^{\dagger}q\left(  \alpha_{ij}\right)  \beta_{j}\right)
+\Tr\left(  \beta_{j}^{\dagger}q\left(  \alpha_{ji}\right)  \beta_{i}\right)  ,
\end{split}
\end{equation}
because $q\left(  \alpha_{ji}\right)  =q\left(  \alpha_{ij}^{\dagger}\right)
=q\left(  \alpha_{ij}\right)  ^{\dagger}$. This proves the claim.
\end{proof}

Every quaternion Hermitian matrix can be diagonalized with symplectic
operations, thus when studying quaternion numerical shadow of Hermitian
matrices, without loss of generality, we can consider only diagonal matrices
with real elements on the diagonal. We note, that for such quaternion matrices,
the representation on a block complex matrices gives us $\nu(A) = A\otimes
\1_2$. Combining this with relation \eqref{prodshad}, we may write the
following chain of equalities
\begin{equation}
\mathcal{P}_{A\otimes \1_4}^{\mathbb{R}} = \mathcal{P}_{A\otimes \1_2} = \mathcal{P}_{A}^{\mathbb{H}}.
\end{equation}
%which is equivalent to 
%\begin{equation}
%\mathcal{D}(a_1,a_1,a_1,a_1, \dots, a_n,a_n,a_n,a_n; \frac12,\frac12,\dots \frac12)
%=\mathcal{D}(a_1,a_1, \dots, a_n,a_n; 1,1,\dots 1)
%=\mathcal{D}(a_1,a_2, \dots, a_n; 2,2,\dots 2)
%\end{equation}

\section{Proofs} \label{sec:proofs}

\begin{proof}[Proof of Lemma \ref{lemma:exp-value}]
Indeed (set $t_{N}:=1-\sum_{i=1}^{N-1}t_{i}$)%
\begin{align}
\mathcal{E}\left[  \left(  1-rX\right)  ^{-\widetilde{k}}\right]    &
=\sum_{n=0}^{\infty}\frac{\left(  \widetilde{k}\right)  _{n}}{n!}%
\int_{\mathbb{T}_{N-1}}\left(  \sum_{i=1}^{N}a_{i}t_{i}\right)  ^{n}%
d\mu_{\mathbf{k}}\nonumber\\
& =\sum_{n=0}^{\infty}\frac{\left(  \widetilde{k}\right)  _{n}}{n!}r^{n}%
\sum_{\alpha\in\mathbb{N}_{0}^{N},\left\vert \alpha\right\vert =n}\binom
{n}{\alpha}\frac{1}{\left(  \widetilde{k}\right)  _{\left\vert \alpha
\right\vert }}%
{\textstyle\prod_{i=1}^{N}}
\left(  k_{i}\right)  _{\alpha_{i}}\nonumber\\
& =\sum_{\alpha\in\mathbb{N}_{0}^{N}}r^{\left\vert \alpha\right\vert }%
{\textstyle\prod_{i=1}^{N}}
\frac{\left(  k_{i}\right)  _{\alpha_{i}}}{\alpha_{i}!}=%
{\textstyle\prod_{i=1}^{N}}
\left(  1-ra_{i}\right)  ^{-k_{i}}.
\end{align}
We used the negative binomial theorem and the multinomial theorem with the
multinomial coefficient $\binom{n}{\alpha}=\frac{n!}{\alpha!}$.
\end{proof}
\medskip

\begin{proof}[Proof of Proposition~\ref{topF}]
The required value is the integral of $d\mu_{\mathbf{k}}$ over the simplex
with vertices $\left(  0,0,\ldots,0\right)  ,\xi_{i}\left(  x\right)  $ for
$1\leq i\leq N-1$. Set $\xi_{i}^{\prime}\left(  x\right)  =\frac{a_{N}%
-x}{a_{N}-a_{i}}$. Change variables to $t_{i}=\xi_{i}^{\prime}\left(
x\right)  s_{i}$, then%
\begin{align}
1-F\left(  x\right)    & =\frac{\Gamma\left(  \widetilde{k}\right)  }%
{\Gamma\left(  k_{N}\right)  }\prod_{i=1}^{N-1}\frac{\xi_{i}^{\prime}\left(
x\right)  ^{k_{i}}}{\Gamma\left(  k_{i}\right)  }\int_{\mathbb{T}_{N-1}}%
{\textstyle\prod_{i=1}^{N-1}}
s_{i}^{k_{i}-1}\left(  1-\sum_{i=1}^{N-1}\xi_{i}^{\prime}\left(  x\right)
s_{i}\right)  ^{k_{N}-1}ds_{1}\ldots ds_{N-1}\\
& =\frac{\Gamma\left(  \widetilde{k}\right)  }{\Gamma\left(  k_{N}\right)
\Gamma\left(  \widetilde{k}-k_{N}\right)  }\prod_{i=1}^{N-1}\xi_{i}^{\prime
}\left(  x\right)  ^{k_{i}}\sum_{\alpha\in\mathbb{N}_{0}^{N-1}}\frac{\left(
1-k_{N}\right)  _{\left\vert \alpha\right\vert }}{\left(  \widetilde{k}%
-k_{N}\right)  _{\left\vert \alpha\right\vert +1}}\prod_{i=1}^{N-1}%
\frac{\left(  k_{i}\right)  _{\alpha_{i}}}{\alpha_{i}!}\xi_{i}^{\prime}\left(
x\right)  ^{\alpha_{i}}. \nonumber
\end{align}
The negative binomial series converges when $0\leq\xi_{i}^{\prime}\left(
x\right)  <1$ for all $i$, that is, $a_{N-1}<x\leq a_{N}$. Now replace
$\xi_{i}^{\prime}\left(  x\right)  $ by $\frac{a_{N}-x}{a_{N}-a_{i}}$ to
obtain the stated formula.
\end{proof}
\medskip

\begin{proof}[Proof of Lemma~\ref{kparfrac}]
The proof follows the method described in Henrici \cite[vol.~1, p.~555]{H}. Use the
notation from equation (\ref{parfrac1}). Set $f_{i}\left(  r\right)  =%
{\displaystyle\prod\nolimits_{j\neq i}}
\left(  1-ra_{j}\right)  ^{-k}$, then%
\begin{equation}%
{\displaystyle\prod\nolimits_{j=1}^{N}}
\left(  1-ra_{j}\right)  ^{-k}=\sum_{j=1}^{k}\frac{\beta_{ij}}{\left(
1-ra_{i}\right)  ^{j}}+q_{i}\left(  r\right)  f_{i}\left(  r\right)  ,
\end{equation}
where $q_{i}\left(  r\right)  $ is a polynomial. Multiply the equation by
$\left(  1-ra_{i}\right)  ^{k}$ to obtain%
\begin{equation}
f_{i}\left(  r\right)  =\sum_{j=1}^{k}\beta_{ij}\left(  1-ra_{i}\right)
^{k-j}+\left(  1-ra_{i}\right)  ^{k}q_{i}\left(  r\right)  f_{i}\left(
r\right)  .
\end{equation}
Apply $\left(  \frac{d}{dr}\right)  ^{m}$ to both sides and set $r=\dfrac
{1}{a_{i}}$. This cancels out every term on the right side except for $j=k-m$;
this term becomes $\beta_{i,k-m}\left(  -a_{i}\right)  ^{m}m!$. By the
generalized product rule%
\begin{equation}
\left(  \frac{d}{dr}\right)  ^{m}f_{i}\left(  r\right)  =\sum_{\alpha
\in\mathbb{N}_{0}^{N},\left\vert \alpha\right\vert =m,\alpha_{i}=0}\binom
{m}{\alpha}\prod_{j=1,j\neq i}^{N}\left(  k\right)  _{\alpha_{i}}a_{j}%
^{\alpha_{i}}\left(  1-ra_{j}\right)  ^{-k-\alpha_{j}}.
\end{equation}
Set $r=\frac{1}{a_{i}}$ to get the stated values of $\beta_{i,k-m}$ (note
$\left(  1-\frac{a_{j}}{a_{i}}\right)  ^{-k-\alpha_{j}}=a_{i}^{k+\alpha_{i}%
}\left(  a_{i}-a_{j}\right)  ^{-k-\alpha_{j}}$).
\end{proof}

\bibliography{shadow_V}
\end{document}